\documentclass[reqno,oneside,11pt]{amsart}

\usepackage{amssymb}
\usepackage[cmtip,all]{xy}
\usepackage[T1]{fontenc} 
\usepackage{pslatex}
\usepackage{enumerate}

\theoremstyle{plain}
\newtheorem{thm}{Theorem}
\newtheorem{prop}[thm]{Proposition}
\newtheorem{lem}[thm]{Lemma}
\newtheorem{lemconj}[thm]{Lemma/conjecture}
\newtheorem{cor}[thm]{Corollary}

\theoremstyle{definition}
\newtheorem{defi}[thm]{Definition}
\newtheorem{exple}[thm]{Example}
\newtheorem{fundamentalexple}[thm]{Fundamental example}

\theoremstyle{remark}
\newtheorem{rem}[thm]{Remark}

\newcommand{\Z}{\mathbb{Z}}

\newcommand{\C}{\mathbb{C}}

\newcommand{\R}{\mathbb{R}}

\newcommand{\F}{\mathbb{F}}
\newcommand{\A}{\mathbb{A}}

\newcommand{\cpo}{\mathbb{P}^1}
\newcommand{\cpd}{\mathbb{P}^2}

\newcommand{\sld}{{\rm SL}(2,\C)}
\newcommand{\psld}{{\rm PSL}(2,\C)}

\newcommand{\geo}{\mbox{Geo}}

\newcommand{\al}{a(\lambda)} 

\newcommand{\am}{a(\mu)} 

\newcommand{\Aut}{\rm{Aut}}
\newcommand{\T}{\mathcal{T}}
\newcommand{\V}{\mathcal{V}}
\newcommand{\lon}{{\rm lg}}
\newcommand{\dist}{{\rm dist}}
\newcommand{\id}{{\rm id}}
\newcommand{\path}{{\mathcal P}}
\newcommand{\dis}{\displaystyle}
\newcommand{\per}{{\mathfrak P} {\mathfrak e} {\mathfrak r}}

\newcommand{\mygraph}[1]{\xybox{\xygraph{#1}}}
\newcommand{\size}{1.2cm}

\newcommand{\dessinUn}{
\mygraph{ !{<0cm,0cm>;<\size,0cm>:<0cm,\size>::}
!{(-2,0)}*{\bullet}="deb" !{(0,0)}*{\bullet}="mil" !{(2,0)}*{\bullet}="fin"
!{(-1,0)}*{\circ}="S1"  !{(1,0)}*{\circ}="S2"
!{(-1.5,-1)}*{\bullet}="B1"  !{(1.5,-1)}*{\bullet}="B2"
"deb"-|*@{>}^{\phi(e)}"S1" "S1"-|*@{>}^{\phi(e')}"mil" 
"mil"-|*@{>}^{\phi(f)}"S2" "S2"-|*@{>}^{\phi(f')}"fin" 
"B1"-|*@{>}_{\phi(g)}"S1" "S2"-|*@{>}_{\phi(h)}"B2" 
}}

\newcommand{\dessinDeux}{
\mygraph{ !{<0cm,0cm>;<\size,0cm>:<0cm,\size>::}
!{(-2,0)}*{\bullet}="deb" !{(0,0)}*{\bullet}="mil" !{(2,0)}*{\bullet}="fin"
!{(-1,0)}*{\circ}="S1"  !{(1,0)}*{\circ}="S2"
!{(-1.5,-1)}*{\bullet}="B1"  !{(1.5,-1)}*{\bullet}="B2"
"deb"-|*@{>}^{\phi(ee')}"S1" "S1"-|*@{>}^{1}"mil" 
"mil"-|*@{>}^{1}"S2" "S2"-|*@{>}^{\phi(ff')}"fin" 
"B1"-|*@{>}_{\phi(ge')}"S1" "S2"-|*@{>}_{\phi(fh)}"B2" 
}}

\newcommand{\dessinTrois}{
\mygraph{ !{<0cm,0cm>;<\size,0cm>:<0cm,\size>::}
!{(-1,0)}*{\bullet}="deb" !{(0,-1)}*{\bullet}="mil" !{(1,0)}*{\bullet}="fin"
!{(0,0)}*{\circ}="S"  
!{(-.7,-.9)}*{\bullet}="B1"  !{(.7,-.9)}*{\bullet}="B2"
"deb"-|*@{>}^{\phi(ee')}"S" "S"-^>(.7){1}"mil"  
"S"-|*@{>}^{\phi(ff')}"fin" 
"B1"-|*@{>}^>(.3){\phi(ge')}"S" "S"-|*@{>}^>(.8){\phi(fh)}"B2" 
}}

\newcommand{\dessinCasUn}{
\mygraph{
!{<0cm,0cm>;<\size,0cm>:<0cm,\size>::}
!{(-3,-1)}*+++{\dessinUn}="deb"
!{(0,+2)}*+++{\dessinDeux}="mil"
!{(+3,-1)}*+++{\dessinTrois}="fin"
"deb"-^{\rm relabel}@{->}@/^3cm/"mil" 
"mil"-^{\rm identify}@{->}@/^3cm/"fin"
}}

\newcommand{\dessinQuatre}{
\mygraph{ !{<0cm,0cm>;<\size,0cm>:<0cm,\size>::}
!{(-2,0)}*{\bullet}="deb" !{(0,0)}*{\bullet}="mil" !{(2,0)}*{\bullet}="fin"
!{(-1,0)}*{\circ}="S1"  !{(1,0)}*{\circ}="S2"
!{(+3,0)}*{\circ}="S3"  !{(1.5,-1)}*{\bullet}="B2"
"deb"-_<u|*@{>}^{\phi(e)}"S1" "S1"-|*@{>}^{\phi(e')}"mil" 
"mil"-_<v|*@{>}^{\phi(f)}"S2" "S2"-|*@{>}^{\phi(f')}"fin" 
"fin"-_<w|*@{>}^{\phi(g)}"S3" "S2"-|*@{>}_{\phi(h)}"B2" 
}}

\newcommand{\dessinCinq}{
\mygraph{ !{<0cm,0cm>;<\size,0cm>:<0cm,\size>::}
!{(-2,0)}*{\bullet}="deb" !{(0,0)}*{\bullet}="mil" !{(2,0)}*{\bullet}="fin"
!{(-1,0)}*{\circ}="S1"  !{(1,0)}*{\circ}="S2"
!{(+3,0)}*{\circ}="S3"  !{(1.5,-1)}*{\bullet}="B2"
"deb"-_<u|*@{>}^{\phi(e)}"S1" "S1"-|*@{>}^{\phi(e')}"mil" 
"mil"-_<v|*@{>}^{\phi(e')^{-1}}"S2" "S2"-|*@{>}^{\phi(e)^{-1}}"fin" 
"fin"-_<w|*@{>}^{s\phi(g)}"S3" "S2"-|*@{>}_{\phi(e'fh)}"B2" 
}}

\newcommand{\dessinSix}{
\mygraph{ !{<0cm,0cm>;<\size,0cm>:<0cm,\size>::}
!{(0,0)}*{\bullet}="deb" !{(0,-2)}*{\bullet}="mil"
!{(0,-1)}*{\circ}="S" 
!{(+1,0)}*{\circ}="S3"  !{(1,-1)}*{\bullet}="B2"
"deb"-_<{u = w}|*@{>}^{\phi(e)}"S" "S"-_>v|*@{>}^>(.6){\phi(e')}"mil"  
"deb"-|*@{>}^{s\phi(g)}"S3" "S"-|*@{>}_{\phi(e'fh)}"B2" 
}}

\newcommand{\dessinCasDeux}{
\mygraph{
!{<0cm,0cm>;<\size,0cm>:<0cm,\size>::}
!{(-3,-1)}*+++{\dessinQuatre}="deb"
!{(0,+2)}*+++{\dessinCinq}="mil"
!{(+3,-1)}*+++{\dessinSix}="fin"
"deb"-^{\rm relabel}@{->}@/^3cm/"mil" 
"mil"-^{\rm identify}@{->}@/^3cm/"fin"
}}

\usepackage[colorlinks, linktocpage, citecolor = magenta]{hyperref}

\begin{document}

\title[Normal subgroup generated by an automorphism]
{Normal subgroup generated by a plane polynomial automorphism}

\author{Jean-Philippe Furter}
\address{Laboratoire MIA \\ Universit\'e de La Rochelle \\
Avenue M. Cr\'epeau\\ 17042 La Rochelle cedex, France}
\email{jpfurter@univ-lr.fr}
\author{St\'ephane Lamy}
\thanks{The second author, on leave from Institut Camille Jordan, Universit\'e Lyon 1, France, was partially supported by an IEF Marie Curie Fellowship.}
\address{Mathematics Institute \\
University of Warwick \\
Coventry CV4 7AL  \\
United Kingdom}
\email{s.lamy@warwick.ac.uk}

\date{March 2010}

\begin{abstract}
We study the normal subgroup $\langle f \rangle_N$ generated by an element $f \neq \id$ in the group $G$
of complex plane polynomial automorphisms having Jacobian determinant $1$.
On one hand if $f$ has length at most $8$
relatively to the classical amalgamated product structure of $G$,
we prove that $\langle f \rangle_N = G$.
On the other hand if $f$ is a sufficiently generic element of even length at least $14$,
we prove that $\langle f \rangle_N \neq G$.
\end{abstract}

\maketitle

\section*{Introduction}

Let Aut$[\C^2]$ denote the group of complex plane polynomial automorphisms
and let $G$ be the subgroup of automorphisms having Jacobian determinant $1$.
In this paper, we deal with normal subgroups of $G$ generated by a single element.

It is easy to check that $G$ is equal to the commutator subgroup of  Aut$[\C^ 2]$
and to its own commutator subgroup as well (see Proposition \ref{prop:comG}).
It is more difficult to decide whether $G$ is a simple group or not.
There  does not seem to exist any natural morphism whose kernel
is a proper normal subgroup of $G$.
However, in a short note published in 1974 that seems to have been quite forgotten,
V. I. Danilov \cite{Da} proves that $G$ is not a simple group.
He uses results from P. Schupp \cite{Sc},
namely the so-called small cancellation theory in the context of an amalgamated product.
To be precise, he shows that the  normal subgroup generated
by the automorphism $(ea)^{13}$ where $a = (y,-x)$ and $e = (x, y+3x^5-5x^4)$
is a strict subgroup of $G$. In fact, he writes $(ea)^{12}$,
because he uses a slightly erroneous definition
of the condition $C'(1/6)$ (see subsection \ref{smallcancellationtheory}).\\

We now introduce the algebraic length of an automorphism in order to state our main result.
The theorem of Jung, Van der Kulk and Nagata  asserts
that  Aut$[\C^2]$ is the amalgamated product over their intersection of the groups $A$ and $E$
of affine and elementary automorphisms (see \ref{generalities}).
Let $f$ be an element of  Aut$[\C^2]$.
If $f$ is not in the amalgamated part $A \cap E$, its algebraic length $|f|$ is defined as
the least integer $m$ such that $f$ can be expressed as a composition $f=g_1 \ldots g_m$,
where each $g_i$ is in some factor ($A$ or $E$) of  Aut$[\C^2]$.
If $f$ is in the amalgamated part, by convention we set $|f| = 0$ (see \cite{Se}, \S 1.3).

The normal subgroup generated by an element $f$ of $G$
will be denoted by $\langle f \rangle_N$.
Of course, $\langle f \rangle_N$ remains unchanged when  replacing $f$ by one of its conjugates in $G$.
So, one can assume  $f$ of minimal algebraic length in its conjugacy class (see subsection \ref{thegroupG} ).
If $|f|\neq 1$, this amounts to saying that $|f|$ is even
(indeed, if $|f|$ is even, it is clear that $f$
is strictly cyclically reduced in the sense of
subsection \ref{smallcancellationtheory} below).
This is for example the case for 
the previous automorphism $(ea)^{13}$ which 
has length 26.\\ 

Here are the two main results of our paper:

\begin{thm} \label{thm1}
If $f \in G$ satisfies $| f | \leq 8$ and $f \neq {\rm id}$,
then  $\langle f \rangle_N = G$.
\end{thm}

\begin{thm} \label{thm2}
If $f \in G$ is  a generic element of even length  $|f| \ge 14$,
then the normal subgroup generated by $f$ in Aut$[\C^2]$
(or a fortiori in $G$) is different from $G$.
\end{thm}

Here the genericness means
that if we write $f^{\pm 1} = a_1e_1\ldots
a_le_l$, where $l \geq 7$,
$a_1, \ldots,a_l \in A \setminus E$
and 
each $e_i=(x+P_i(y),y)$, then there exists an
integer $D$ such that for any sequence
$d_1,\ldots,d_l$
of integers $\geq D$,  $(P_1,\ldots,P_l)$ can be chosen
generically
(in the sense of algebraic geometry,
i.e. outside a Zariski-closed hypersurface)
in the affine space ${\dis \prod_{1 \,
\leq \, i \, \leq \, l} \C [y]_{\leq \, d_i}}$,
where we have set $\C [y]_{\leq \, d}= \{ P \in \C [y];
\deg P \leq d \}$.

Theorems \ref{thm1} and \ref{thm2} correspond in the
text  below
to Theorems \ref{thm:A} and \ref{thm:B}.
Note that in the latter statements
we use a geometric notion of length
coming from Bass-Serre theory (see subsection \ref{lengths}).
This geometric length allows us to obtain more natural statements.
In fact, Theorem \ref{thm:B} deals with automorphisms
satisfying the special condition $(C2)$
(see Definition  \ref{def:C2}).
The proof that this condition is indeed generic is postponed to the annex.
To convince the reader that such a condition is necessary,
we now give  examples of automorphisms of arbitrary even
length and generating normal subgroups equal to $G$. 

\begin{exple}
Consider the three automorphisms
$$a  =  (y,-x), \quad  e_1  = (x+P(y),y), \quad e_2 =
(x+Q(y),y),$$
where $P$ (resp. $Q$) is an even (resp. odd) polynomial of
degree $\geq 2$,
and  set $f=ae_1(ae_2)^n$, where $n \geq 1$ is an
integer.
If $u= - \id$, we get $au=ua$, $e_2u=ue_2$ and
$e_1u=ue_1^{-1}$,
so that the commutator $[f,u]=fuf^{-1}u^{-1}$ is equal to
$$[f,u]=ae_1(ae_2)^nu
(ae_2)^{-n}(ae_1)^{-1}u^{-1}=ae_1ue_1^{-1}a^{-1}u^{-1}
=ae_1^2a^{-1}.$$
Since $[f,u] \in \langle f \rangle_N$, we get $e_1^2 \in
\langle f \rangle_N$,
so that $\langle f \rangle_N = G$ by  Theorem \ref{thm1} (or
by Lemma \ref{lem:length1} below).
\end{exple}

One motivation for this work is the still open question of
the simplicity of the Cremona group ${\rm Cr}_2$,
i.e. the group of birational transformations of $\C^2$.
For instance in \cite{Gi}
the question is explicitly stated and Gizatullin gives
several criterion
that would prove that ${\rm Cr}_2$ is simple.
Recently Blanc \cite{Blanc} proved that ${\rm Cr}_2$
is simple as an infinite dimensional algebraic group. 
In this respect, we should mention that Shafarevich claimed
that the group ${\rm Aut}_1 [\C^n]$ of automorphisms of
the affine space $\C^n$ having Jacobian determinant $1$
is simple as an infinite dimensional algebraic group
for any $n \geq 2$
(see \cite[Th. 5]{Sha1} and \cite[Th. 5]{Sha2}).
However, it is known that these two papers contain some inaccuracies
(see \cite{Ka1,Ka2}), so the status of this question is not clear to us.

After studying the polynomial case,
our opinion is that ${\rm Cr}_2$, view as an abstract group, could be not simple as well.
Indeed, it is known since Iskovskikh \cite{Is} that ${\rm Cr}_2$
admits a presentation as the quotient
of an amalgamated product by the normal subgroup generated
by a single element.
Take $H_1 =({\rm PGL}(2) \times {\rm PGL}(2)) \rtimes
\Z/2\Z$
the group of birational transformations that extend as
automorphisms of $\cpo \times \cpo$
and take $H_2$ the group of transformations that preserve the
pencil of vertical lines $x = cte$.
Take  $\tau = (y,x) \in H_1 \setminus
H_2$ and $e = (1/x,y/x) \in H_2 \setminus H_1$;
then ${\rm Cr}_2$ is equal to the quotient
$$ \left( H_1 *_{H_1 \cap H_2} H_2 \right)  /_{\langle f
\rangle_N}$$
where $f =  (\tau e)^3$.
To prove that ${\rm Cr}_2$ is not simple it would be
sufficient to find an element $g$
in the amalgamated product of $H_1$ and $H_2$
(that should correspond to a sufficiently general birational transformation)
such that the normal subgroup $\langle f,g \rangle_N$ is proper.
This is  similar to the results we obtain in this paper; but the problem seems harder in the birational setting.

As a final remark on these matters,
we would like to mention a nice reinterpretation of Iskovskikh's result by Wright (see \cite[Th. 3.13]{Wr2}).
Let $H_3 = {\rm PGL}(3)$ be the group of birational transformations that extend as
automorphisms of $\cpd$.
Then Wright proves that the group ${\rm Cr}_2$ is the free product of $H_1,H_2,H_3$ amalgamated
along their pairwise intersection in ${\rm Cr}_2$. \\

In this paper we chose to work over the field $\C$ of
complex numbers, even if most of the results could be
adapted to any base field. Note that in the case of a
finite field the nonsimplicity result is almost immediate.
Let $\F_q$ denote  the finite field of $q=p^n$ elements,
where $p$ is prime and $n \geq 1$.
Let $\Aut [\F_q^2]$ be the  group of automorphisms 
of the affine plane $\A^2_{\F_q}= \F_q^2$ and let $\Aut_1
[\F_q^2]$
be the normal subgroup of automorphisms having Jacobian
determinant $1$.
If $X$ is a finite set, let $\per (X)$ (resp. $\per ^+(X)$)
be the group of permutations (resp.  even permutations) of
$X$.
Since the natural morphism
$\phi :  \Aut [ \F_q^2] \to \per ( \F_q^2)$
induces a non-constant morphism
$\Aut_1[ \F_q^2] \to \per ( \F_q^2)$ (consider the
translations!),
it is clear that $\Aut_1[ \F_q^2]$ is not simple.

\begin{rem}
If $q$ is odd (i.e. the characteristic $p$ of $\F_q$ is
odd),
one can easily check that
$\phi ( \Aut_1[ \F_q^2])= \per ^+ ( \F_q^2)$.
Indeed, $\phi$ is surjective (see \cite{Ma}),
so that $\phi ( \Aut_1[ \F_q^2])$ is a normal subgroup of
$\per ( \F_q^2)$.
However, if the cardinal of $X$ is different from $4$,
it is well known that $\per ^+ (X)$ is the only
nontrivial normal subgroup of $\per  (X)$ (see e.g.
\cite[ex. 3.21, p. 51]{Rot}).
Therefore, it is enough to show that $\phi ( \Aut_1[
\F_q^2]) \subseteq \per ^+ ( \F_q^2)$.
But on one hand $\Aut_1[ \F_q^2]$ is generated by the
elementary automorphisms
$(x+P(y), y)$ and $(x,y +Q(x))$, where $P \in \C[y]$, $Q\in
\C[x]$ are any polynomials.
On the other hand, it is straightforward to check that such
 automorphisms
induce even permutations of $\F_q^2$. \\
\end{rem}

As a final remark, we would like to stress the importance
of translations in getting our results. Let  $\Aut^0[ \C^2]$
be the group of automorphisms fixing the origin
and let $J_n$ be the natural group-morphism associating to
an element of
$\Aut^0[ \C^2]$ its $n$-jet at the origin (for $n \geq 1$).
For $n \geq 2$, the kernel of $J_n$ is a nontrivial normal
subgroup of $G^0= G \cap \Aut^0[ \C^2] $,
so that this latter group is not simple. Of course for $\Aut
[ \C ^2 ]$ the morphism $J_n$ does not exist.
This explains the fact that our paper strongly relies on translations
(see  Lemmas \ref{subtreefixedbyatranslation} and
\ref{stabilizer}).

\begin{rem}
It results from \cite{Anick} that the image of $J_n$ is
exactly
the group of $n$-jets of polynomial endomorphisms fixing the
origin
and whose Jacobian determinant is a non-zero constant.
The precise statement can be found in \cite[Proposition
3.2]{Fu}.\\
\end{rem}

Our paper is organized as follows. 

In section \ref{sec:bass-serre} we gather the results from
Bass-Serre theory that we need:
this includes some basic definitions and facts but also some
quite intricate computations,
such as in the characterization of tripods (subsection
\ref{independentcolors}).
This is also the place where we define precisely the
condition $(C2)$ that we need in Theorem \ref{thm:B}.

Section \ref{sec:proofoftheorem1} is devoted to the proof of
Theorem \ref{thm1}.
This is the most elementary part of the paper.
We only use Lemma \ref{subtreefixedbyatranslation} from
section \ref{sec:bass-serre}.   

In section \ref{Lyn-Sc} we deal with  R-diagrams.
This field of combinatorial group theory has been
introduced 
by Lyndon and Schupp in relation with condition $C'(1/6)$
from small cancellation theory
(see \ref{smallcancellationtheory}).
A noteworthy feature of our work is that we use R-diagrams
in a completely opposite setting (positive curvature).

In section \ref{sec:proofoftheorem3}  we are able to give a
proof of Theorem \ref{thm2},
using the full force of both Bass-Serre and Lyndon-Schupp
theories.

We briefly discuss in section \ref{sec:10-12} the
cases not covered
by Theorems \ref{thm1} and \ref{thm2},
that is to say when the automorphism has length 10 or 12.

Finally, in the annex, we  prove that condition $(C2)$ is generic
and we also give explicit examples of automorphisms satisfying this condition.

\section{The Bass-Serre tree}\label{sec:bass-serre}

\subsection{Generalities}\label{generalities}

The classical theorem of Jung, van der Kulk and Nagata
states that the group $\Aut[\C^2]$
is the amalgamated product of the \textbf{affine} group  
  $$A = \left\lbrace  
  (\alpha x + \beta y + \gamma, \delta x + \epsilon y +
\zeta); \alpha, \ldots, \epsilon \in \C,
 \alpha \epsilon - \beta \delta \neq 0 \right\rbrace  $$
and  the \textbf{elementary} group  
  $$E = \left\lbrace  (\alpha x + P(y), \beta y + \gamma);
\alpha, \beta, \gamma \in \C, \alpha \beta \neq 0, P \in
\C[y] \right\rbrace $$
over their intersection (see \cite{Ju,vdK,Na}).
This is usually written in the following way:

\begin{thm} \label{Jung}
$\Aut[\C^2]= A *_{A \cap E} E$.
\end{thm}

A geometric proof of this theorem and many references may be
found in \cite{LaJung}.
Let us also recall that elements of $E$ are often called
\textbf{triangular} automorphisms.\\

The Bass-Serre theory (\cite{Se}) associates
a simplicial tree to any amalgamated product.
In our context, let us denote by $\T$ this tree.
By definition, the vertices of  $\T$ are the disjoint union
of the left cosets modulo $A$ (vertices of  {\sl{type $A$}})
and modulo $E$ (vertices of {\sl{type $E$}}).
The edges of $\T$ are  the left cosets modulo $(A \cap E)$.
Finally, if $\phi \in \Aut[\C^2]$,  the edge $\phi (A \cap
E)$
links the vertices $\phi A$ and $\phi E$.
Since $\Aut[\C^2]$ is generated by $A$ and $E$, $\T$ is
connected.
Thanks to the amalgamated structure, $\T$ contains no loop,
so that it is indeed a tree.\\

The group $\Aut[\C^2]$ acts naturally on $\T$ by left
multiplication:
for any $g, \phi \in \Aut[\C^2]$, we set
$g. \phi A= (g \phi )A$, $g.  \phi E= (g \phi) E$ and $g.
\phi (A \cap E)= (g \phi) (A \cap E)$.
It turns out that this action gives an embedding of
$\Aut[\C^2]$
into the group of simplicial isometries of $\T$ (see
\cite[Remark 3.5]{LaAlg}).
This action is transitive on the set of edges,
on the set of vertices of type $A$ 
and on the set of vertices of type $E$. 
The stabilizer of a vertex $\phi A$
({\sl{resp.}} of a vertex $\phi E,$ {\sl{resp.}} of an edge
$\phi (A\cap E)$) is the group
$\phi A \phi^{-1}$ ({\sl{resp.}} $\phi E\phi^{-1},$
{\sl{resp.}} $\phi (A \cap E) \phi^{-1}$).\\

Following \cite{Wr,LaAlg}, one can define systems of
representatives
of the nontrivial left cosets $A/A\cap E$ and $E/A\cap E$
by taking:
\begin{eqnarray*}
\al &=& (\lambda x+y, -x); \hspace{3mm}  \lambda \in \C \\
e(Q) &=& (x+Q(y),y); \hspace{3mm}   Q(y) \in y^2 \C [y]
\setminus \{ 0 \}.
\end{eqnarray*}
Note that the minus sign in the expression of $\al$ did not
appear in  \cite{Wr,LaAlg}.
We have to introduce it in the present paper
in order to get automorphisms with Jacobian determinant 1
(see subsection \ref{thegroupG}).

Then any element $g \in \Aut[\C^2]$ may be uniquely written
$g=ws$
where $w$ is a product of factors  
of the form $\al$ or  $e(Q)$,
successive factors being of different forms,
and $s \in A \cap E$ (see e.g. \cite[chap. I,
1.2, th. 1]{Se}).
Similarly, any edge
(resp. vertex of type $A$, resp. vertex of type $E$)
may be uniquely written $w(A\cap E)$
(resp $wA$, resp. $wE$) where $w$ is as above.

We call a (directed) \textbf{path} a sequence of consecutive edges in
$\T$.
To denote a path we enumerate its vertices separated by $-$.
For instance the path $\path$ of two edges containing the
vertices $\id A, \id E, eA$,
where $e \in E \setminus A$
will be denoted $\path = \id A - \id E - eA$.
If we are only interested in the type of the vertices,
we say for example that $\path$ is of type $A-E-A$.\\

If two vertices of $\T$ are fixed by an automorphism of
$\Aut [\C^2]$,
then the path relating them is also fixed.
Therefore, the subset of $\T$ fixed by an automorphism is
either empty or a subtree.
Up to  conjugation, this subset has been computed for any
automorphism in
\cite[proof of Proposition 3.3]{LaAlg}.
In particular, it has been computed for the translation
$(x+1,y)$.
The following easy and technical lemma is a slight variation
of this computation.
As in the latter paper, this analogous statement turns out
to be very useful.
The proof is given for the sake of completeness.

\begin{lem} \label{subtreefixedbyatranslation}
The subtree of $\T$ fixed by the translation $(x + c, y)$,
$c \in \C^*$,
is exactly the union of the paths
$$\id E - e(P) A -e(P) \al E- e(P) \al  e(Q) A$$
where $P \in y^2 \C[y]$, $\lambda \in \C$ and  $Q(y)= \alpha
y^2$, $\alpha \in \C^*$.
\vskip1mm
Note that we (exceptionally) allow $P$ to be zero.
In that case, the path should rather be written
$$\id E - \id A - \al E-  \al  e(Q) A.$$
In particular, the fixed subtree
does not depend on $c$,
has diameter $6$
and contains the closed ball of radius 2 centered at $\id
E$,
i.e. the union of the paths
$$\id E - e(P) A -e(P) \al E,
\hspace{3mm} P \in y^2 \C[y],
\hspace{3mm} \lambda \in \C.
$$
\end{lem}

\begin{proof} 
If $P,Q \in y^2 \C[y]$ and $\lambda \in \C$ we have
\begin{eqnarray*}
(x+c,y) \circ e(P) &=& e(P) \circ (x+c,y);\\
(x+c,y) \circ \al &=& \al \circ (x,y+c);\\
(x,y+c) \circ e(Q)  &=& e(Q) \circ f;
\end{eqnarray*}
where $f=(x+Q(y)-Q(y+c), y+c)$, so that
\begin{eqnarray*}
(x+c, y) e(P) \al e(Q) = e(P) \al e(Q) f.
\end{eqnarray*}
Therefore, the vertex $e(P) \al e(Q) A$ is fixed by
$(x+c,y)$
if and only if $f \in A$, i.e. ${\rm deg}(Q(y)-Q(y+c)) \leq
1$,
i.e. ${\rm deg}(Q) \leq 2$.
If $Q= \alpha y^2$, this vertex is fixed.
Since  the vertex $\id E$ is also (obviously) fixed, this
shows that
the following path is fixed:
\begin{eqnarray*}
\id E - e(P) A -e(P) \al E- e(P) \al  e(Q) A.
\end{eqnarray*}
If $Q= \alpha y^2$, where $\alpha \neq 0$ and $\mu \in \C$,
it remains to show that
the vertex $e(P) \al  e(Q) \am E$ is not fixed.
Indeed, an easy computation shows that
\begin{eqnarray*}
(x+c, y) e(P) \al e(Q) \am = e(P) \al e(Q) \am g,
\end{eqnarray*}
where $g=(x-c, 2 \alpha c x + y + \mu c - \alpha c^2) \notin
E$. 
\end{proof}

\subsection{Algebraic and geometric lengths} \label{lengths}

We will use two notions of \textbf{length} on $\Aut[\C^2]$.

The  {\bf algebraic length} has been defined in the
introduction:
if $g \in \Aut[\C^2]$ is not in the amalgamated part,
$|g|$ is defined as the least integer $m$ such that
$g$ can be expressed as a composition
$g=g_1 \ldots g_m$ where each $g_i$ is in some factor of the
amalgam.
If $g$ is in the amalgamated part, we set $|g|=0$.

The {\bf geometric length} is defined by ${ \lon (g) =
\inf_{v \in { \V} } \dist (g.v,v) }$,
where $\V$ is the set of vertices of $\T$ and
$\dist(.,.)$ is the simplicial distance on $\T$. 

By Lemma \ref{strictlyandweaklycyclicallyreducedelements} we almost always have 
$\lon (g) = \min\{|\phi g \phi^{-1}|; \phi \in {\rm Aut}[\C^2]\}$,
the only exception being when $g$ is conjugate to an elementary automorphism
which is not conjugate to an element in the amalgamated part.

\subsection{Elliptic and hyperbolic elements} \label{ellipticandhyperbolicelements}

Elements $g$ of $\Aut[\C^2]$ may be sorted into two classes
according to their action on  $\T$.

If $\lon (g) = 0$ (i.e. $g$ has at least one fixed point on 
$\T$),
we say that  $g$ is  {\bf elliptic}.
This corresponds to the case where $g$ is conjugate to an
element
belonging to some factor ($A$ or $E$) of $\Aut[\C^2]$.
Since any element of $A$ is conjugate to some element of
$E$,
this amounts to saying that $g$ is triangularizable
(i.e. conjugate to some triangular automorphism).

If $\lon (g) > 0$, we say that $g$ is {\bf hyperbolic}.
This corresponds to the case where $g$ is conjugate to
a composition of generalized H\'enon transformations $h_1
\ldots h_l$ (see \cite{FM}).
We recall that a generalized H\'enon transformation is a map
of the form

\centerline{$h=(y, ax+P(y))= (y,x) \circ (ax+P(y), y)$,}

\noindent where $a \in \C^*$ and $P(y)$ is a polynomial of
degree at least $2$.
Equivalently, $g$ is conjugate to an automorphism of the
form

\centerline{$f=a_1e_1 \ldots a_le_l$,}

\noindent where each $a_i \in A \setminus E$
and each $e_i \in E \setminus A$.

The set of points $v \in \T$  satisfying $\dist(g.v,v)=\lon
(g)$
defines an infinite geodesic of $\T$ denoted by $\geo(g)$.
Furthermore,  $g$ acts on $\geo(g)$ by translation of length
$\lon(g)$.
It is not difficult to check that $\lon(g)= \lon (f) = |f| = 2l$
and that the geodesic of $f$ is composed of the path
${\rm id}A-a_1 E-a_1e_1A- \cdots -a_1e_1\ldots a_le_lA$
and its translated by the $f^k$'s ($k \in \Z$).
If $g= \phi f \phi ^{-1}$ with $\phi \in {\rm Aut}[\C^2]$, we have of course
$\geo(g)= \phi ( \geo(f) )$.\\

The proof of the following easy result is left to the
reader. Note that these two sets of equivalent conditions correspond
to the notions of strictly and weakly cyclically reduced elements given
in subsection \ref{smallcancellationtheory}.

\begin{lem}   \label{strictlyandweaklycyclicallyreducedelements}
Let $g \in \Aut[\C^2]$ be a hyperbolic element.
\begin{enumerate}
\item The following assertions are equivalent:

\centerline{(i) $|g| = \lon(g) $; \hspace{5mm}
(ii) $\geo(g)$ contains the vertices $\id A$ {\bf and}  $\id
E$.}

\vskip1mm

\item The following assertions are equivalent:

\centerline{(iii) $  |g| \leq \lon(g) +1 $; \hspace{5mm}
(iv) $\geo(g)$ contains the vertex $\id A$ {\bf or} $\id
E$.}
\end{enumerate}
\end{lem}

\subsection{The group $G$} \label{thegroupG}

In this subsection we prove two basic facts about $G$.
Let us set $A_1=A \cap G$ and $E_1=E \cap G$.
Theorem \ref{Jung} easily implies
the following result:

\begin{prop}\label{prop:jungG}
$G= A_1 *_{A_1 \cap E_1} E_1$.
\end{prop}

\begin{proof}
By \cite[chap. I, n$^{\circ}$ 1.1, Prop. 3]{Se},
it is sufficient to prove that any $g \in G$ is a composition
of affine and triangular automorphisms
with Jacobian determinant 1.
We know that we can write $g$ as a composition of $\al$ and
$e(Q)$,
with a correcting term $s \in A\cap E$.
Note that the $\al$ and $e(Q)$ are automorphisms with
Jacobian determinant 1,
so $s$ is also of Jacobian determinant 1 and we are done. 
\end{proof}

As a consequence of this proposition the whole discussion
of the previous subsection  still applies to $G$.
In particular we can make the same choice of representatives
$\al$ and $e(Q)$ to write edges and vertices,
so that there exists a natural bijection between the trees
associated to $\Aut[\C^2]$ and to $G$.

\begin{prop}\label{prop:comG}
The group $G$ is the commutator subgroup of the group
Aut$[\C^2]$,
and is also equal to its own commutator subgroup.  
\end{prop}

\begin{proof} Using Proposition \ref{prop:jungG},
it is sufficient to check that the commutator subgroup of
$G$ contains $\sld$
and all triangular automorphisms of the form $(x+P(y), y)$.
But on one hand it is well-known that $\sld$ is equal to its
own commutator subgroup;
on the other hand any triangular automorphism $(x+ \lambda
y^n,y)$,
with $n \geq 2$ and $\lambda \in \C$, is the commutator
of $(x+\lambda (1-b) ^{-1} y^n,y)$ and $\left( bx,b^{-1}y
\right)$,
where $b \neq 1$ is a $n$-th root of the unity.
Finally, any translation $(x+c,y)$ is the commutator of
$(-x,-y)$ and $(x-\frac {c}{2},y)$.  
\end{proof}

\subsection{The color} \label{color}

We now introduce the {\bf color} of a path of type
$A-E-A$.
This notion will be used to make precise the genericness
assumptions we need.
Note that any path of type $A-E-A$ can be written $\path=
\psi e_1 A - \psi E - \psi e_2 A$,
where $\psi \in {\rm Aut}[\C^2]$ and $e_1,e_2 \in E$.

\begin{defi} 
\label{def:color} 
The \textbf{color} of $\path$ is the double
coset
$(A\cap E) e_1^{-1}e_2 (A\cap E)$.
\end{defi}

One verifies easily that this definition does not depend on
the choice of $e_1, e_2$.
The color  is clearly invariant under the  action of Aut$[\C^2]$.
In fact, given two paths of type $A-E-A$ one could even show
that one can send one to the other (by an element of Aut$[\C^2]$)
if and only if they have the same color.
However, we will not use this result.
As an illustration of the notion of color, we can 
note that the color of the path
$e(P) A - e(P) a( \lambda ) E - e(P) a( \lambda ) e (Q) A$
appearing in Lemma  \ref{subtreefixedbyatranslation}
has color $(A\cap E) e(Q) (A\cap E)$.

If $P \in \C[y]$ is such that the color of $\path$ is equal to
the double coset  $(A\cap E) e(P) (A\cap E) $,
we say that $P$ represents the color of $\path$.  
The following lemma implies that
this notion does not depend  on the orientation of the path.
Its proof is easy and left to the reader.

\begin{lem} \label{obviousresult}
Let $P,Q \in \C[y]$ be polynomials of degree $\geq 2$.
Then $P$ and $Q$ represent the same color if and only if there exist 
$\alpha, \ldots, \epsilon$ with $\alpha \beta \neq 0$ such that
$Q(y)= \alpha P ( \beta y + \gamma) + \delta y + \epsilon$.
\end{lem}

\begin{rem} \label{rem:13}
Note that any path of type $A-E-A$ can be sent by an automorphism
to a path of the form $\id A- \id E - e(P) A$.
It is easy to check that the vertices $e(P) A$ and $e(Q)A$
are equal if and only if there exists
$\alpha, \beta \in \C$ such that $Q(y) = P(y) + \alpha y + \beta$.\\
\end{rem}

\begin{fundamentalexple}
\label{exple:fundamental}
Let $g$ be a hyperbolic automorphism  of geometric length $\lon(g) = 2l$.
We know that $g$ is conjugate to an automorphism of the form
$f=a_1e_1 \ldots a_le_l$ where each $a_i \in A \setminus E$
and each $e_i \in E \setminus A$.
Then, the geodesic of $g$ (and  $f$) carries the $l$ colors
$(A \cap E) e_i (A \cap E)$ ($1 \leq i \leq l$) which are repeated periodically.   
\end{fundamentalexple}

\subsection{General color} \label{generalcolor}

\begin{defi}
\label{def:general}
A polynomial $P \in \C [y]$  of degree $d \geq 5$ is said to be
\textbf{general} if it satisfies:

$\forall \, \alpha, \beta, \gamma \in \C, 
\hspace{3mm} 
\deg ( P(y) - \alpha P ( \beta y + \gamma)) \le d-4
\hspace{3mm} \Longrightarrow \alpha = \beta =1 \mbox{ and }
\gamma=0$.

The color $(A \cap E) e(P) (A \cap E)$
is said to be \textbf{general} if $P$ is general.
Lemma \ref{obviousresult} implies that
this notion does not depend on the choice of a representative $P$.
\end{defi}

\begin{lem} \label{stabilizer} Let $Q \in y^2 \C[y]$ be
general.
The stabilizer of the path $\path=e(Q)A - \id E-\id A$ is
equal to
$\{ (x+ \beta y+ \gamma,y); \beta, \gamma \in \C \}$.
Furthermore, if $\beta \neq 0$,
the automorphism $(x+ \beta y + \gamma ,y)$
does not fix any path strictly containing $\path$.
\end{lem}

\begin{proof}
We know that $f\in \Aut [\C^2 ]$ fixes the path
$\id E-\id A$
if and only if $f \in A \cap E$.
In this case, there exists constants $\alpha, \ldots,
\zeta$,
with $\alpha \varepsilon \neq 0$ 
such that
$f= ( \alpha x + \beta y + \gamma, \varepsilon y + \zeta)$.
Since $f e(Q)= e(Q) g$, where
$g= (\alpha x + \beta y + \alpha Q (y) - Q( \varepsilon y +
\zeta), \varepsilon y + \zeta)$,
the vertex $e(Q) A$ is fixed by $f$ if and only if $g \in
A$,
i.e. $\deg (\alpha Q (y) - Q( \varepsilon y + \zeta)) \leq
1$.
The polynomial $Q$ being general, this is equivalent
to $\alpha = \varepsilon=1$ and $\zeta=0$.

The second assertion comes from the following simple
observation: 
$$(x+ \beta y + \gamma ,y) \al E = a(\lambda - \beta )E. $$
Indeed, since $(x+ \beta y + \gamma ,y) e(Q) =e(Q) (x+ \beta y + \gamma, y)$,
we also have
$$(x+ \beta y + \gamma ,y) e(Q) \al E = e(Q) a(\lambda - \beta )E.$$
Therefore, the vertices $\al E$ and $e(Q) \al E$ are  fixed
by $(x+ \beta y + \gamma ,y)$ if and only if $\beta=0$. 
\end{proof}

\begin{rem} \label{subtreefixedbyatranslationbis}
Lemma \ref{stabilizer} is a kind of converse to Lemma
\ref{subtreefixedbyatranslation}.
Precisely, we obtain that if $\phi$ fixes a general path of
4 edges centered on $\id E$,
then $\phi = (x+c,y)$ (Here by general we mean that the color supported
by the two central edges of the path is general; see Def. \ref{def:color} and below).

Note also that since $(x,y+c)= a(0) \circ  (x-c,y) \circ 
a(0)^{-1}$,
the subset of $\T$ fixed by  $(x, y +c )$
is the image by $a(0)$ of the subset fixed by $(x-c,y)$.
In particular, it contains the closed ball of radius 2
centered at $a(0) E$.
Furthermore, if $\phi$ fixes a general path of 4 edges
centered at $a(0) E$,
it can be written as $\phi = (x,y+c)$.\\
\end{rem}

We now apply the notion of a general color to prove a technical result
that we need to prove Theorem \ref{thm:B}. 
We consider a hyperbolic automorphism $f$
and $g = \varphi f \varphi^{-1} \neq f$ a conjugate of $f$.
We want to show that if $f$ is sufficiently general
then $\geo (f) \cap \geo (g)$ is a path of length at most 4.
More precisely, we also describe all possibles types of such paths.

\begin{defi} \label{def:C1}
We say that a hyperbolic automorphism of geometric length $2l$ satisfies condition $(C1)$
if the $l$ colors supported by its geodesic (see Example \ref{exple:fundamental}) are general and distinct.
\end{defi}

In the annex we show that
this condition is  generic in a natural sense.

\begin{prop}\label{prop:geo}
Let $f$ and $g = \phi f \phi^{-1}$ be two distinct conjugate
automorphisms
satisfying condition $(C1)$.
If the intersection $\geo(f) \cap \geo(g)$ contains at least
one edge
then this path is of type:
$$ A-E, \; E-A-E, \; A-E-A,\mbox{ or } E-A-E-A-E $$
\end{prop}

\begin{proof}

There is no restriction to assume that $\path ' = \geo(f) \cap \geo(g) =\geo(f) \cap  \phi ( \geo(f) )$ contains
a path of type $A-E-A$, because otherwise
$\path '$ is at most a path of type $E-A-E$.

Let us call $v$ the central vertex of type $E$ of this subpath of $\path '$.
Since $\phi^{-1}(v) \in \geo (f)$, there exists an integer $k$
such that $\dist (f^k (v), \phi ^{-1} (v) ) = \dist ( (\phi f^k) (v), v) < \lg(f) = 2l$.
Replacing $\phi$ by $\phi f^k$, we do not change $g$, but we now have
$\dist ( \phi (v), v) <  2l$.
By  condition $(C1)$, the geodesic of $f$ carries $l$
distinct colors which are repeated periodically.
Therefore, $\dist(\phi(v), v) \in 2l\Z$ and finally we get
$\phi (v) =v$, so that $\phi$ is elliptic.

Let us set   $\path = \phi ^{-1} ( \path ' ) = \geo(f) \cap \phi^{-1}(\geo(f))$.
Equivalently, one may define $\path$ as the maximal path
such that $\path \subseteq \geo(f)$ and $\phi( \path ) \subseteq \geo(f)$.

The path $\path$ contains a path of type $A-E-A$ whose central vertex is $v$.
Without loss of generality, one can now conjugate and assume
that this  subpath is of the
form $e(Q)A- \id E - \id A$. In particular $v = \id E$.

There are two subcases, depending on whether $\phi\colon \path \to \phi(\path)$
preserves the orientation induced by $\geo(f)$.

If $\phi$ preserves this orientation, then $\phi$ fixes
$\path$ point by  point.
We may assume that $\path$ is strictly greater than
$e(Q)A- \id E - \id A$, because otherwise there is nothing to show.
Then, by  Lemmas \ref{subtreefixedbyatranslation} and \ref{stabilizer}, we 
get $\phi= (x+ \gamma, y)$.
Since the colors of $\geo (f)$ are general, Lemma \ref{subtreefixedbyatranslation}
shows us that $\path$ is of the form
$e(Q) a(\lambda) E - e(Q) A- \id E - \id A - a( \mu) E$,
so that it is of type $E-A-E-A-E$.

If $\phi$ does not preserve this orientation, then $\phi$
fixes only the vertex $v$ of $\geo(f)$.
One can show that $\phi$ has to be an involution (see Lemma
\ref{lem:involution} below).
This implies that $\path$ contains an even number of edges
and is centered on $v$.
Since the $l$ colors supported by $\geo (f)$ are distinct,
$\path$ contains only one color, so that it is of type
$A-E-A$ or $E-A-E-A-E$. 
\end{proof}

\begin{lem}\label{lem:involution}
Let $\path$ be a path of type $A-E-A$ carrying a general
color.
If $\phi \in {\rm Aut}[\C^2]$ exchanges the two ends of $\path$ then 
$\phi^2=\id$. 
\end{lem}

\begin{proof}
Without loss of generality, one can conjugate
and assume that the path $\path$ is of the form $e(Q)A-\id
E-\id A$ (see Remark \ref{rem:13}).
Note that $\phi_1 = e(Q) \circ (-x,y)$ is an involution
that exchanges the two vertices $e(Q)A$ and $\id A$.
Thus $\phi_1 \phi$ fixes the path $\path$ point by point,
and since $Q$ is general by Lemma \ref{stabilizer} we get
$\phi = \phi_1 \circ (x+\beta y+\gamma,y)$.
Remark that $\phi_1 \circ  (x+\beta y+\gamma,y)=(x+\beta
y+\gamma,y)^{-1} \circ \phi_1$, hence

\vskip1mm
\hspace{20mm} $\phi^2 = \phi_1 \circ  (x+\beta y+\gamma,y) \circ (x+\beta
y+\gamma,y)^{-1} \circ \phi_1 = \id.$   
\end{proof}

\begin{exple}
Here we show that all cases allowed by Proposition
\ref{prop:geo} can be realized.
In the following examples we suppose that $\geo(f)$ contains
the path
$a(0)E-\id A-\id E-e(Q)A-e(Q)\am E$ where $Q$ is a general polynomial
and we choose $\phi$ such that the path $\path$ has various
forms. 

\begin{enumerate}
 \item Examples with $\phi$ fixing  at least one edge:
\begin{itemize}
\item $\phi = (x+P(y),y)$ with $\deg P \geq 2$,
$\path = { \id A}-{\id E}$;
\item $\phi = (\alpha x,\beta y)$ with $\alpha \beta \neq 0$ and $(\alpha, \beta ) \neq (1,1)$,
$\path = {a(0)E}-{ \id A}-{\id E}$;
\item $\phi = (x+by,y)$ with $b \neq 0$, $\path = {\id A}-{\id E}-{e(Q)A}$;
 \item $\phi = (x+c,y)$ with $c \neq 0$, $\path =a(0)E-\id A-\id E-e(Q)A-e(Q)\am E $.\\
\end{itemize}

\item Examples with $\phi$ reversing the orientation:
\begin{itemize}
 \item $\phi = (y,x)$ exchanges $a(0)E$ and $\id E$,
 $\path = {a(0)E}-{ \id A}-{\id E} $;
 \item $\phi = (-x+Q(y),y)$ exchanges $\id A$ and $e(Q)A$,
 $\path$ is of length 4 or 2 depending if $\mu = 0$ or
not.\\
\end{itemize}

\item  Example with $\phi$ hyperbolic:
\begin{itemize}
 \item $\phi = e(Q)\am u$ with $u = (-x,-y)$ sends $\path =
a(0)E  - \id A - \id E$
to $\phi( \path ) = \id E - e(Q)A-e(Q)\am E$
 (the reader should verify that $\am u a(0) = (x-\mu y,y)
\in A\cap E$).
\end{itemize}
\end{enumerate}
\end{exple}

\subsection{Independent colors and tripods} \label{independentcolors}

\begin{defi} \label{independentsequence}
A family of polynomials $P_i \in \C [y]$ ($1\leq i \leq l$)
is said to be \textbf{independent} if
given any $\alpha_k, \beta_k, \gamma_k \in \C$ with
$\alpha_k \beta_k \neq 0$
and $i_k \in \{ 1, \ldots, l \}$, for $1 \leq k \leq 3$, we
have:

$$ \deg \hspace{-1mm} \sum_{1 \,  \leq  \, k \,  \leq  \, 3} 
\hspace{-3mm} \alpha _k P_{i_k} ( \beta_k y + \gamma_k)
\hspace{2mm} \leq \hspace{2mm} 1 \hspace{5mm} 
\Longrightarrow \hspace{5mm} i_1=i_2=i_3.$$

The family of colors $(A \cap E) e(P_i) (A \cap E)$ ($1\leq i \leq l$)
is said to be \textbf{independent} if the family  $P_i$ ($1\leq i \leq l$)
is independent.
Lemma \ref{obviousresult} implies that
this notion does not depend on the choice of the representatives $P_i$.
\end{defi}

\begin{defi}
Three paths $\path _1$, $\path _2$, $\path _3$ of the tree $\T$
define a \textbf{tripod} if 
\begin{itemize}
 \item For each $i \neq j$, $\path _i \cap \path _j$ contains at least one edge; 
 \item The intersection $\path _1 \cap \path _2 \cap \path _3$ consists of exactly one vertex $v$.
\end{itemize}

The three paths $\path _i \cap \path _j $ are called the
\textbf{branches} of the tripod.
The vertex $v$ is called the \textbf{center} of the tripod.

If we have a center of type $E$, we can consider
the three colors associated with the three  paths of type
$A-E-A$
containing the center and included in the tripod.
In this situation we say that any one of these colors
is a \textbf{mixture} of the two other ones.

\end{defi}

\begin{lem} \label{algebraiccriterionofmixture}
Let $P_1,P_2,P_3 \in \C[y]$ be polynomials of degree $\geq 2$.
The following assertions are equivalent:
\begin{enumerate}
\item
$(A \cap E) e(P_3) (A \cap E)$ is a mixture of the 
$(A \cap E) e(P_i) (A \cap E)$'s ($1 \leq i \leq 2$);
\item
$\exists \, \alpha_1,\beta_1, \gamma_1,\alpha_2,\beta_2, \gamma_2, \delta, \epsilon  \in \C$ with
$\alpha_1 \beta_1 \alpha_2 \beta_2 \neq 0$
such that
$$P_3(y)= \alpha_1 P_1( \beta_1 y + \gamma_1) + \alpha_2 P_2( \beta_2 y + \gamma_2)
+ \delta y + \epsilon.$$
\end{enumerate}
\end{lem}

\begin{proof}
(1) $\Longrightarrow$ (2).
Assume that there exists a tripod admitting the 3 colors
$(A \cap E) e(P_i) (A \cap E)$ ($1 \leq i \leq 3$).

We may assume that the center of this tripod is $\id E$ 
and that one of its branch is $\id E - \id A$.
Let $\widetilde{P}_1, \widetilde{P}_2 \in \C[y]$ be such that the 2 other branches
are
$\id E - e(\widetilde{P}_1 )A$, and $\id E - e(\widetilde{P}_2)A$ and such that
$(A \cap E) e(P_1) (A \cap E) = (A \cap E) e(\widetilde{P}_1) (A \cap E)$
and $(A \cap E) e(P_2) (A \cap E) = (A \cap E) e(\widetilde{P}_2) (A \cap E)$.
By Lemma \ref{obviousresult}, for $1 \leq i \leq 2$,
there exists $\alpha_i,\beta_i, \gamma_i,\delta_i, \epsilon_i$
with $\alpha_i \beta_i \neq 0$
such that
$\widetilde{P}_i= \alpha_i P_i( \beta_i y + \gamma_i) + \delta_i y + \epsilon_i$.

We then have $(A \cap E) e(P_3) (A \cap E) = (A \cap E) e(\widetilde{P}_3) (A \cap E)$,
where $\widetilde{P}_3=\widetilde{P}_1-\widetilde{P}_2$,
so (still by Lemma \ref{obviousresult})
this shows that $P_3$ has the desired form.\\

(2) $\Longrightarrow$ (1).
Set $\widetilde{P}_1= \alpha_1 P_1( \beta_1 y + \gamma_1)$,
$\widetilde{P}_2= - \alpha_2 P_2 ( \beta_2 y + \gamma_2)$
and 
$\widetilde{P}_3 = \widetilde{P}_1 - \widetilde{P}_2 =
\alpha_1 P_1( \beta_1 y + \gamma_1) + \alpha_2 P_2 ( \beta_2 y + \gamma_2)$.
By Lemma \ref{obviousresult},
we have
$(A \cap E) e(\widetilde{P}_i)  (A \cap E) = (A \cap E) e(P_i) (A \cap E)$
for $1 \leq i \leq 3$.
Since $ e( \widetilde{P}_2 )^{-1} e( \widetilde{P}_1 )= e( \widetilde{P}_3 ) \notin A$,
the vertices $ e( \widetilde{P}_1 )A$ and  $e( \widetilde{P}_2 )A$ are distinct.
Consider the tripod with center $\id E$
and branches $\id E - \id A$, $\id E - e( \widetilde{P}_1 )A$
and $\id E - e( \widetilde{P}_2 )A$.
Its three colors are
$(A \cap E) e(\widetilde{P}_i ) (A \cap E)$ for $1 \leq i \leq 3$.
This shows that
$(A \cap E) e(P_3) (A \cap E)$ is
a mixture of
$(A \cap E) e(P_1) (A \cap E)$ and $(A \cap E) e(P_2) (A \cap E)$. 
\end{proof}

\begin{rem} \label{symmetricform}
The second condition of Lemma \ref{algebraiccriterionofmixture}
may be written under the following symmetric form:

For $1 \leq k \leq 3$, there exists $\alpha_k, \beta_k, \gamma_k \in \C$ with $\alpha_k \beta_k \neq 0$
such that
$$\deg \hspace{-1mm} \sum_{1 \,  \leq  \, k \,  \leq  \, 3}
\alpha_k P_k (\beta_k y + \gamma_k) \hspace{2mm}  \leq  \hspace{2mm}1.$$
Therefore, the following lemma is an easy consequence of
the previous one.
\end{rem}

\begin{lem}
Consider three colors represented by  $P_1,P_2,P_3 \in
\C[y]$ which are  polynomials of degree $\geq
2$.
The following assertions are equivalent:
\begin{enumerate}
 \item the three colors $(A \cap E) e(P_i) (A \cap E)$
($i=1,2,3$) are independent;
 \item For any $i_1,i_2,i_3 \in \{ 1,2,3 \}$, if
$(A \cap E) e(P_{i_3}) (A \cap E)$ is a mixture of
$(A \cap E) e(P_{i_1}) (A \cap E)$ and $(A \cap E)
e(P_{i_2}) (A \cap E)$,
then $i_1=i_2=i_3$.
\end{enumerate}
\end{lem}

\begin{defi} \label{def:C2}
We say that a hyperbolic automorphism of geometric length $2l$ satisfies condition $(C2)$
if the $l$ colors supported by its geodesic (see Example \ref{exple:fundamental}) are general and independent.
\end{defi}

In the annex we show that
this condition is  generic in a natural sense.

\begin{rem} \label{C1->C2} One could easily check that independent colors are necessarily distinct.
Therefore, condition $(C2)$ is stronger than condition $(C1)$.\\
\end{rem}

By misuse of language, we will say that three hyperbolic automorphisms $g_1, g_2, g_3$
define a \textbf{tripod} if their geodesics $\geo (g_1), \geo(g_2), \geo (g_3)$
define a tripod.

\begin{lem}\label{lem:tripod}
A tripod associated with three conjugates of a hyperbolic automorphism $f$
satisfying  condition $(C2)$ admits branches of length at most 2.
\end{lem}

\begin{proof}
If the center of the tripod is of type $A$, by Proposition
\ref{prop:geo} there is nothing to do.
Assume now that the center of the tripod is of type $E$.
Without loss of generality one can conjugate and assume that
the center is $\id E$,
and that $\geo(f)$ contains the vertices $\id A$ and
$a(0)E$.
We denote by $g=ufu^{-1}$ and $h=vfv^{-1}$ the two conjugates
of $f$ involved in the tripod.

By condition $(C2)$ the three colors  centered on $\id E$ in
the tripod must be equal.
Indeed, if $(A \cap E) e(P_i) (A \cap E)$, $1 \leq i \leq l$ are the $l$ colors
supported by $\geo (f)$, then there exist $i_1,i_2,i_3 \in \{ 1, \ldots, l \}$
such that these three colors are $(A \cap E) e(P_{i_k}) (A \cap E)$, $1 \leq k \leq 3$.
By Definition \ref{independentsequence}  and Lemma \ref{algebraiccriterionofmixture}
(see also Remark \ref{symmetricform}), we get $i_1=i_2=i_3$, so that the three colors are equal.

Let us prove  that $u$ can be chosen fixing the center $\alpha = \id E$ of the
tripod. Since $\alpha \in \geo (f) \cap \geo (g) = \geo (f) \cap u ( \geo (f) )$,
we get $u^{-1}( \alpha ) \in \geo (f) $, so that there exists an integer $k$
such that $\dist ( f^k ( \alpha ) , u ^{-1} ( \alpha ) ) < \lon (f) =2l$.
Replacing $u$ by $u f^k$, we do not change $g$, but we now have
$\dist ( u ( \alpha ), \alpha ) < 2l$.
By  condition $(C1)$ (cf. Remark \ref{C1->C2}), the geodesic of $g$ carries $l$
distinct colors which are repeated periodically.
Therefore, $\dist(u ( \alpha ), \alpha ) \in 2l\Z$ and finally we get
$u( \alpha ) =\alpha$.
We would prove in the same way that $v$ can be chosen fixing $\alpha = \id E$.
In other words, we have $u,v \in E$.

Let us now assume that there exists a branch, say $\geo(f)
\cap \geo(h)$,
of length strictly greater than $2$.
Then, by Proposition \ref{prop:geo}, this branch has length $4$,
with middle point $a(0) E$ (see Fig. \ref{fig:tripod}).
Since $v$ fixes point by point the general path $\geo(f)
\cap \geo(h)$,
by Remark \ref{subtreefixedbyatranslationbis},
it can be written as $v= (x,y+c)$.

Let $e = e(P) = (x+P(y),y) \in E$ be such that
the vertex $eA \in \geo (f) \cap \geo (g)$.
Since $\geo (h) = v ( \geo (f) )$, the vertex $veA \in \geo
(h)$
and finally $veA \in \geo (g) \cap \geo (h)$.

We assume that the orientation induced by $g$ on $\id E-eA$
is opposite to the one of $f$, the other case being 
symmetric.
 
\begin{figure}[ht]
 $$\xygraph{
!{<0cm,0cm>;<1cm,0cm>:<0cm,1cm>::}
!{(-1,3)}*{\bullet} ="eamuE" 
!{(-2,4)}*{\bullet} ="eamueQA" 
!{(0,2)}*{\bullet} ="e1A" 
!{(0,0)}*{\bullet} ="e2A" 
!{(1,1)}*{\bullet} ="idE" 
!{(2.5,1)}*{\bullet} ="idA" 
!{(4,1)}*{\bullet} ="B1" 
!{(5.5,1)}*{\bullet} ="B2"
!{(7,1)}*{\bullet} ="B3"
!{(2.5,1.5)}="fi" 
!{(0.5,2)}="ff"
!{(-0.5,0.5)}="gf"
 !{(-0.5,1.5)}="gi"
!{(0.5,0)}="hf" 
!{(2.5,.5)}="hi"
"e1A"-_<{eA}"idE" 
"e2A"-^<{veA}"idE" 
"idA"-_>(0.9){\id E}_<{ \id A}"idE"
"B1"-^>{a(0)e(Q)A}"B2" "B2"-_>{a(0)e(Q)a(\lambda)E}"B3" 
"idA"-^>(0.9){a(0)E}"B1" 
"eamueQA"-"eamuE" 
"eamuE"-"e1A"
"fi"-@{.>}@/^0.2cm/_f"ff"
"hi"-@{.>}@/_0.2cm/_h"hf"
"gi"-@{.>}@/^0.2cm/_g"gf"
}$$
\caption{}
\label{fig:tripod}
\end{figure}

%Hence, we have $g = \phi f \phi^{-1}$ with $\phi \in E$.
Since $\geo (g)= u ( \geo (f) )$,
$u$ sends the path $\id A - \id E- e A$
to the path $e A - \id E - veA$.

On one hand, $u$  sends $\id A$ to $eA$, i.e. $uA=eA$, i.e.
$e^{-1}u \in A$,
i.e. $e^{-1}u \in A \cap E$.
Since $e^{-1}u \in A \cap E$, it can be written as $s_1s_2$,
where $s_1 = (a_1 x, b_1 y + c_1), s_2 = (x + \beta y +
\gamma,y) \in A \cap E$
and we have $u = es_1s_2$.

On the other hand $u$  sends $eA$ to $veA$,
i.e. $ueA=veA$, i.e. $es_1s_2eA=veA$.
Since $s_2e=es_2$, we have $es_1s_2eA=es_1eA$,
so that $es_1eA=veA$.
This last equality is still equivalent to 
$e^{-1}v^{-1}es_1e \in A$.
We compute
$$ e^{-1}v^{-1}es_1e = (a_1 x + a_1 P(y) +P(b_1 y + c_1) -
P(b_1y +c_1-c),b_1y +c_1-c).$$
We should have $\deg ( a_1 P(y) +P(b_1 y + c_1) - P(b_1y
+c_1-c)) \leq 1$.
Since $a_1 \neq 0$ and $\deg (P(b_1 y + c_1) - P(b_1y
+c_1-c)) < \deg P$,
this is impossible.  
\end{proof}

\section{The proof of Theorem \ref{thm1}} \label{sec:proofoftheorem1}

We start by looking at the case of an automorphism of algebraic length
$\leq 1$,
i.e. a triangular or affine automorphism.
Note that similar results in the context of birational
transformations are proved in \cite{Gi} and \cite{CD}.

\begin{lem}\label{lem:length1}
If $f \in G$ satisfies $| f | \leq 1$ and $f \neq {\rm id}$,
then  $\langle f \rangle_N = G$.
\end{lem}

\begin{proof}
Let $g,h \in G$. Note that if $g$ or $h$ belongs to $\langle f
\rangle_N$, then so does the commutator $[g,h] = ghg^{-1}h^{-1}$.
We show that $G = \langle f \rangle_N$ by making the
following observations:

\begin{itemize}
 \item If $f \in \sld$ and $f \neq \pm \id$, we obtain $\sld
\subseteq \langle f \rangle_N$.
We used the fact that $\{ \pm \id \}$ is the unique nontrivial normal subgroup of $\sld$.
Indeed, if $H$ is a normal subgroup of $\sld$ not included
into $\{ \pm \id \}$,
we get $\sld = H \cup (-H)$ by simplicity of $\psld$.
Therefore, if $g=(y,-x)$, we get $g \in H$ or $-g \in H$,
so that $-\id= g^2=(-g)^2 \in H$ and finally $H=\sld$.

Now, if $\alpha, \beta \in \C$, we get
$$ [(x+\alpha, y+ \beta) , (-x,-y)] = (x+ 2 \alpha, y + 2
\beta) $$
so that $A \subseteq \langle f \rangle_N$.

If $b \neq 1$ is a $n$-th root of the unity ($n \geq 2$) and
$\lambda \in \C$, we get
$$ [(x+ \lambda (1-b)^{-1}y^n, y) , (bx,b^{-1}y)] = (x+
\lambda y^n,y)$$
and we are done.

\vskip1mm

\item If $f$ is a translation, then, conjugating by $\sld$,
 we see that $\langle f \rangle_N$ contains all 
translations.
So, it contains  the commutator
$$[(x,y+1), (x+y^2,y)] = (x-2y+1,y)$$
and also the linear automorphism $(x-2y,y)$.
We conclude by the previous case.

\vskip1mm

\item If $f$ is an  affine automorphism which is not a
translation,
then there exists a translation $g$ which does not commute
with $f$.
Therefore, the commutator $[f,g]$ is a nontrivial
translation
belonging to $\langle f \rangle_N$ and we conclude by the
previous case.

\vskip1mm

\item Finally if $f = (ax+P(y),a^{-1}y+c)$ is a triangular
non affine automorphism,
then, up to replacing $f$ by $[f,g]$,
where $g$ is a triangular automorphism non commuting with
$f$, we may assume that
$a=1$.
Still replacing  $f$ by $[f,g]$,
where $g$ is a triangular automorphism non commuting with
$f$,
we may even assume that $c=0$.
Therefore, $f$ is of the form $(x+P(y),y)$.
Remark then that the  commutator 
$$[(x,y+1),(x+P(y),y)]$$ 
is a triangular automorphism of the form $(x+R(y),y)$, with
$\deg R = \deg P -1$.
By induction on the degree we obtain the existence of a
nontrivial translation $(x+c,y)$
in $\langle f \rangle_N$.
This case has already been done. 
\end{itemize} 
\end{proof}

\begin{cor} \label{cor:elliptic}
If $f \in G$ is elliptic (i.e. triangularizable) and $f \neq {\rm id}$,
then  $\langle f \rangle_N = G$.
\end{cor}

We are now ready to prove Theorem \ref{thm1}.
In fact, we will prove the following stronger and more
geometric version:

\begin{thm} \label{thm:A}
If $f \in G$ satisfies $\lon (f) \leq 8$ and $f \neq {\rm
id}$,
then  $\langle f \rangle_N = G$.
\end{thm}

\begin{proof} 
The crucial fact we use here is the knowledge of the subtree
fixed by translations  $(x+c,y)$.
We know that this subtree is of diameter 6, centered in $\id
E$,
and that the closed ball of radius 2 and center $\id E$ is
contained in this subtree
(see Lemma \ref{subtreefixedbyatranslation}).
In consequence, given an arbitrary path of type $E-A-E-A-E$,
there exists a conjugate $\psi$ of  $(x+1,y)$  fixing this
path point by point. 

Let us choose such a path contained in the geodesic of $f$
and let us set $g= \psi f \psi^{-1}$. Then if $\lon (f) = 2$
or $4$ it is clear that $f \circ g^{-1}$ is  elliptic,
so we can conclude by Corollary \ref{cor:elliptic}. If $\lon (f)
= 6$
then $ \lon (f \circ g^{-1} ) \le 4$ so we are done by the
previous case.

The case where $\lon (f) =8$ is more subtle and we  have to
refine the above argument.
Replacing $f$ by one of its conjugates, we may assume $|f| =
\lon(f) =  8$. We can then assume
(maybe replacing $f$ by $f^{-1}$) that 
$$f = e_1 a_1 e_2 a_2 e_3 a_3 e_4 a_4$$
where $a_i \in A\setminus E$, $e_j \in E\setminus A$.
Without loss of generality
we can further assume that each $e_j$ is of the form $e_j =
e(P_j) = (x+P_j(y), y)$
and that $\deg(e_1) \le \deg(e_j)$ for $j = 2,3,4$.

We know that any translation $(x+c, y)$ fixes the closed
ball of radius 2 and center $\id E$.
Note also that for any $s \in A \cap E$,
$s (x+1,y) s^{-1}$ is still a translation of the form
$(x+c,y)$. 
In consequence, if we write  $e_1a_1$ under the form $e_1a_1
= e(P)\al s$ with $s \in A \cap E$, the automorphism 
\begin{eqnarray*}
\tilde{e}_1 &=&  e_1 a_1 (x+1,y) a_1^{-1} e_1^{-1} \\
&=&  (x+P(y),y) \circ (\lambda x+y,-x) \circ (x+c, y) \circ
(-y,\lambda y+x) \circ (x-P(y),y)\\
&=&  (x+\lambda c+P(y-c)-P(y),y-c)
\end{eqnarray*}
fixes the closed ball of radius 2 and center $e_1a_1E$.
Note that ${\rm deg}\, \tilde{e}_1= {\rm deg}\,  e_1 -1$.
Consider  
$$g  =  \tilde{e}_1 f \tilde{e}_1^{-1} \mbox{ and } h =
g^{-1}  f.$$
By construction the geodesics $\geo(g)$ and $\geo(f)$ have
at least 4 edges in common.
By Lemma \ref{subtreefixedbyatranslation} we also know that
they have at most 6 edges in common.
Then we can check (see Fig. \ref{fig:long8})
that $h$ sends the vertex $v= a_4^{-1}e_4^{-1}a_3^{-1}E$
to a vertex at distance at most 8 (and at least 6) of $v$.
Explicitly, one can compute 
$$h = \tilde{e}_1 a_4^{-1} e_4^{-1} a_3^{-1}  \tilde{e}_3 
a_3 e_4 a_4 $$
where $ \tilde{e}_3  = e_3^{-1} a_2^{-1}  (x-1,y) a_2 e_3$
is a triangular automorphism
with $\deg(\tilde{e}_3 ) = \deg(e_3) -1$.

If $\deg(\tilde{e}_1 ) = \deg(\tilde{e}_3) =1$ then $\lg(h)
= 4$.
This corresponds to the case when $\geo(g)$ and $\geo(f)$
share 6 edges.
Note that $a_3^{-1}  \tilde{e}_3  a_3$ and $a_4  \tilde{e}_1
 a_4^{-1}$
are indeed non triangular affine automorphisms.

If $\deg(\tilde{e}_1 ) = 1$ and $\deg(\tilde{e}_3) \ge 2$
then $\lg(h) = 6$.
In this case $\geo(g)$ and $\geo(f)$ share 5 edges : the
vertices $\id A$ and $\tilde{e}_1A$ coincide.

In the two cases above we are done by the  first part of the
proof.

Finally if $\deg(\tilde{e}_1 ) \ge 2$
then $h$ admits a factorization similar to the one of the
$f$ we started with 
except that the first triangular automorphism has a strictly
smaller degree.
By induction, we can produce an element of length $8$ in
$\langle f \rangle_N$ 
with the first triangular automorphism of degree $2$, and we
are done by the previous argument.  
\end{proof}

\begin{figure}
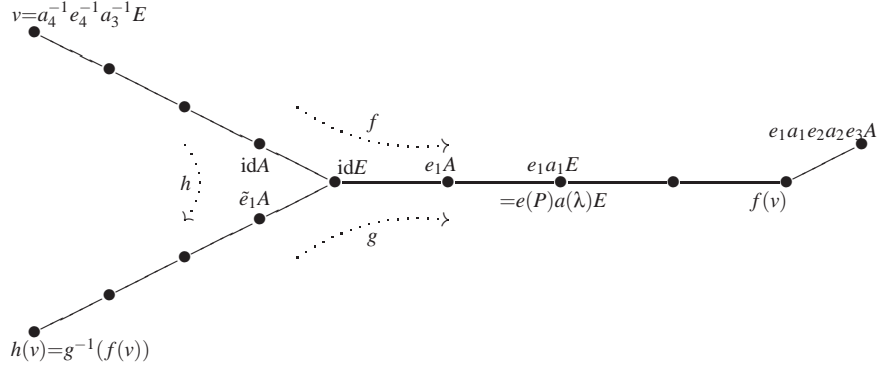

$$\xygraph{
!{<0cm,0cm>;<1cm,0cm>:<0cm,1cm>::}
!{(-1,2)}*{\bullet} ="S3" 
!{(-2,2.5)}*{\bullet} ="S2"
 !{(-3,3)}*{\bullet} ="S1" 
!{(0,1.5)}*{\bullet} ="idA"
!{(0,0.5)}*{\bullet} ="e1tildeA" 
!{(1,1)}*{\bullet} ="idE"
!{(2.5,1)}*{\bullet} ="e1A" 
!{(4,1)}*{\bullet} ="e1a1E"
!{(5.5,1)}*{\bullet} ="S4"
!{(7,1)}*{\bullet} ="S5"
!{(8,1.5)}*{\bullet} ="S55"
!{(-1,0)}*{\bullet} ="S6"
!{(-2,-0.5)}*{\bullet} ="S7"
!{(-3,-1)}*{\bullet} ="S8"
!{(2.5,1.5)}="ff"
!{(0.5,2)}="fi"
!{(-1,0.5)}="hf"
!{(-1,1.5)}="hi"
!{(0.5,0)}="gi"
!{(2.5,.5)}="gf"
"idA"-_<{\id A}"idE"
"e1tildeA"-^<{\tilde{e}_1A}"idE" 
"e1A"-_>(0.9){ \id E}_<{e_1A}"idE" 
"e1a1E"-_<{e_1a_1E}^<{=e(P)\al E}"e1A" 
"S1"-^<{v = a_4^{-1}e_4^{-1}a_3^{-1}E}"S2" 
"S2"-"S3"
"S4"-_>(0.9){f(v)}"S5" 
"S5"-^>{e_1a_1e_2a_2e_3A}"S55"
"S6"-"S7"
"S8"-_<{h(v)=g^{-1}( f(v))}"S7"  
"S3"-"idA"
"e1a1E"-"S4"
"e1tildeA"-"S6" 
"fi"-@{.>}@/_0.2cm/^f"ff"
"hi"-@{.>}@/^0.2cm/_h"hf"
"gi"-@{.>}@/^0.2cm/_g"gf"
}$$ 
\caption{Proof of Theorem \ref{thm:A}}
\label{fig:long8}
\end{figure}

\section{R-diagrams} \label{Lyn-Sc}

\subsection{Generalities on small cancellation theory}
\label{smallcancellationtheory}

In this subsection we consider $H=H_1 *_{H_1\cap H_2} H_2$ a
general
amalgamated product of two factors.
Of course our motivation is to apply the theory to the group
Aut$[\C^2]$ of plane automorphisms.

The following definitions are taken from \cite{LS}, chap. V,
\S 11 (p. 285).
If $u$ is an element of $H$, not in the amalgamated part
$H_1\cap H_2$, a {\bf normal form} of $u$ is any sequence
$x_1 \cdots x_m$ such that $u=x_1 \cdots x_m$,
each $x_i$ is in a factor of $H$,
successive $x_i$ come from different factors of $H$,
and no $x_i$ is in the amalgamated part.
The {\bf length} of $u$ is defined by
$|u |= m$.
This definition does not depend on the chosen normal form,
but only on $u$.
If $u$ is in the amalgamated part of $H$, by convention we
set  $|u|=0$.

We call a \textbf{word} an element $u\in H$ given with a
factorization $u = u_1 \cdots u_k$, where $u_i \in H$ for $i
= 1, \cdots,k$.  
A word $u = u_1 \cdots u_k$ is said to have {\bf reduced
form} if $|u_1 \cdots u_k|= |u_1| + \cdots + |u_k|$.

Suppose $u$ and $v$ are elements of $H$ with normal forms
$u=x_1 \cdots x_m$ and $v=y_1 \cdots y_n$.
If $x_m y_1$ is in the amalgamated part,
we say that there is \textbf{cancellation} between $u$  and
$v$
in forming the product $uv$.
Equivalently, this means that $|uv| \leq |u| + |v| -2$.
If $x_m$ and $y_1$ are in the same factor of $H$
and $x_my_1$ is not in the amalgamated part, we say that
$x_m$ and $y_1$ are
\textbf{consolidated} in forming a normal form of $uv$.
Equivalently, this means that $|uv| = |u| + |v| -1$.

A word is said to have {\bf semi-reduced form} $u_1 \cdots
u_k$
if there is no cancellation in this product.
Consolidation is expressly allowed.

A word $u=x_1 \cdots x_m$ in normal form is {\bf strictly}
(resp. {\bf weakly})
{\bf cyclically reduced}\label{def:reduced} if $m \leq 1$
or if $x_m$ and $x_1$ are in different factors of $H$
(resp. the product $x_m x_1$ is not in the amalgamated part).
These two notions correspond to 
the two sets of equivalent conditions given in 
Lemma   \ref{strictlyandweaklycyclicallyreducedelements}

A subset $R$ of $H$ is {\bf symmetrized} if all elements of
$R$ are weakly cyclically reduced
and for each $r \in R$, all weakly cyclically reduced
conjugates of both $r$ and $r^{-1}$ belong to $R$.

If $f$ is strictly cyclically reduced, $R(f)$ denotes the
symmetrized set generated by $f$,
i.e. the smallest symmetrized set containing $f$. It is
clear that $R(f)$ is equal
to the set of conjugates of $f^{\pm 1}$ of length $\leq | f
| + 1 $.\\

We now discuss briefly the condition $C'(\lambda)$
(mostly used with $\lambda = 1/6)$.
We do not need this notion in our construction,
but this was the original setting where the notion of
R-diagram (see next subsection) was introduced. 
Let $R$ be a symmetrized subset of $H$.
A word $b$ is said to be a \textbf{piece} (relative to $R$)
if there exists distinct elements
$r_1,r_2$ of $R$ such that $r_1=bc_1$ and $r_2=bc_2$ in
semi-reduced form.

\begin{lem} \label{lem:C'1/6}
If $ 0 < \lambda < 1$ and $\forall\, r \in R$, $|r| > 1/
\lambda$, the following assertions are equivalent:
\begin{enumerate}
 \item  If $r \in R$ admits a  semi-reduced form $r=bc$,
where $b$ is a piece of $R$, then $| b | < \lambda | r | $;
 \item $\forall \, r_1,r_2 \in R$ such that $r_1 r_2 \neq
1$,
$| r_1 r_2 | > | r_1 | + | r_2| - 2 \lambda \min \{ | r_1|,
|r_2 | \} + 1$.
\end{enumerate}
\end{lem}

\begin{proof} The equivalence is easily obtained from the following
claim.

Let $r_1=bc_1$ and $r_2=bc_2$ be semi-reduced expressions
with $b \neq 1$ and $r_1 \neq r_2$.

\vskip1mm

\noindent \underline{Claim.} There exists $b', c'_1, c'_2$
such that:
\begin{enumerate}[a)]
\item  the equalities $r_1=b'c'_1$ and $r_2=b'c'_2$ hold;
\item these expressions are semi-reduced;
\item exactly one of these expressions is reduced;
\item the expression $(c'_1)^{-1}c'_2$ is reduced;
\item $| b' | \geq | b |$.
\end{enumerate}  
\end{proof}

\begin{defi}
When the equivalent assertions of Lemma \ref{lem:C'1/6} are
satisfied,
we say that $R$ satisfies condition $C'(\lambda)$.
\end{defi}

The first assertion is the one used by Lyndon and Schupp.
The second one is used by Danilov,
except that he forgets the $+1$ in the formula.
This leads to the slight error in his statement that we
mentioned in the introduction.
Let us finish this subsection by recalling one of the main theorems
of small cancellation theory  (see \cite[Th. 11.2, p. 288]{LS}):

\begin{thm}
Let $R$ be a symmetrized subset of the amalgamated group $H$.
Suppose that $R$ satisfies condition $C'( \lambda)$ with
$\lambda \leq 1/6$, then the normal subgroup generated by $R$
in $H$ is different from $H$.
\end{thm}

\subsection{Construction of an R-diagram}

The  idea of associating diagrams in the Euclidean plane to
some products in amalgamated groups appears
in  \cite{vK}. 

In 1966, Lyndon independently arrived at the same idea
and Weinbaum rediscovered van Kampen's paper
(see \cite{Ly,We} and \cite{LS}, p. 236).
For the basic definition of a \textbf{diagram}, we refer to
\cite{LS}, chap. V, \S 1, p. 235.
Here follows a quick review of this notion.

A diagram is a plane graph (or more generally a graph on an
orientable surface,
we will consider spherical diagrams in Lemma
\ref{lem:courbure}). 
Vertices are divided into two types, \textbf{primary} and
\textbf{secondary}.
Any edge joining two vertices gives rise to two
directed edges (according to the chosen directions) which we call \textbf{half-segments}.
If $e$ denotes one of these half-segments, $e^{-1}$ will refer
to the other one (obtained by reversing the direction of $e$).
The notation 'edge' will be used later on to refer to some
special unions of half-segments (see the remark on terminology
below).
A half-segment will always join vertices of different types.
By definition,  \textbf{segments} will denote some special successions of
two half-segments that we now describe.
If $e_1,\ldots,e_r$ are the half-segments arriving at some secondary vertex $v$
and taken counterclockwise, then, by definition, the segments passing through $v$
are the successive half-segments $e_i, e_{i+1}^{-1}$  and their inverses $e_{i+1}, e_i^{-1}$,
for $1 \leq i \leq r$, where $i$ and $i+1$ are taken modulo $r$.
If two successive half-segments $e,e'$ 
define a segment, the latter will be noted $ee'$.
Note that the initial and terminal vertices of
a segment have to be  primary.
By convention, each segment (resp. half-segment) has length $1$
(resp. $1/2$).
Each oriented half-segment $e$ will be labeled by an element
$\phi (e)$
belonging to a factor of Aut$[\C^2]$,
with the labels on successive half-segments at a secondary
vertex
belonging to the same factor.
The identity $ \phi (e^{-1})= \phi (e) ^{-1}$ is required.
This labelling gives a
labelling on segments, by taking $\phi(ee') =
\phi(e)\phi(e')$.
The label on an individual half-segment may be in the
amalgamated part,
but if $e,e'$ are the two half-segments of a segment, we
will usually insist that $\phi (ee')$ is not in the
amalgamated part
(in fact, there will be only one exception to this rule,
see step 4 in the proof of Theorem
\ref{constructionofM}).
We call \textbf{region} a bounded connected component of the
complement of the graph in the surface.
A \textbf{boundary cycle} of a region $D$ is a collection of
half-segments that run along the entire boundary of $D$
(say counterclockwise in the case of the plane, or in a way
compatible with the orientation in general)
with initial vertex of primary type. Similarly, a boundary
cycle of the diagram is a collection of half-segments
that run along the boundary of the diagram.
Let us note that a segment necessarily belongs to the boundary of
some region and/or to the boundary of the diagram.

Now let $f$ be an element of Aut$[\C^2]$ and consider $R(f)$
the associated symmetrized set.
We say that a diagram is a \textbf{\boldmath $R(f)$-diagram} if for
any region $D$ and any boundary cycle $e_1 \ldots e_s$ of $D$,
we have $ \phi (e_1) \ldots \phi (e_s) \in R(f)$.\\

\noindent{\bf Terminology.} Note that we use two kinds of
graph in this paper: the Bass-Serre tree and the diagrams
of Lyndon-Schupp.
In the context of the Bass-Serre tree we have already used
the term \textit{edge},
and we have called a \textit{path} the union of several
edges.
In the context of  Lyndon-Schupp diagrams, we have
\textit{segments} and \textit{half-segments}.
We call \textit{edge} in this context a connected component
of the intersection of the boundary of two regions,
which is a collection of half-segments.\\

The following result will be the key ingredient for the
proof of  Theorem  \ref{thm2}.
Its proof will occupy the rest of this subsection.

\begin{thm} \label{constructionofM}
Let $f \in G$ be a strictly cyclically reduced element of $G$
of (even) algebraic length $|f| \ge 2$.
Assume that  the normal subgroup generated by $f$ in
Aut$[\C^2 ]$ is equal to $G$.
Then there exists a planar $R(f)$-diagram $M$ such that:
\begin{enumerate}
 \item \label{point(i)} $M$ is connected and simply
connected;
 \item \label{point(ii)} The boundary of $M$ has length
$\frac12$ or $1$;
 \item \label{point(iii)} If $e_1 e_1' \ldots e_te_t'$ is
a boundary cycle of some region of $M$,
then $t= |f|$ and $\phi (e_1e_1') \ldots \phi (e_te_t')$ is
a reduced form
of a strictly cyclically reduced conjugate of  $f$.
\end{enumerate}
\end{thm}

\begin{proof}
We start by choosing an element $g \neq \id$ with
$\lon(g) = 0$. By assumption we can write  
$$g= ( \phi_1 f^{\pm 1} \phi_1^{-1}) \cdots ( \phi_n f^{\pm
1} \phi_n^{-1}).$$
with $\phi_i \in {\rm Aut}[\C^2]$.

We assume that we have chosen $g$ such that $n$ is minimal.
By Lemma \ref{reducedform},
we may assume that
each $\phi_i f^{\pm 1} \phi_i^{-1}$ is expressed under
reduced form
$\psi_i r_i \psi_i^{-1}$
(i.e. $|\psi_i r_i \psi_i^{-1}| = |\psi_i| + | r_i|  +
|\psi_i^{-1}|$)
where $r_i \in R(f)$.
There is no restriction to assume that $| \psi_i|=0$ if and
only if $\psi_i=\id$.
Note also that the four following assertions are equivalent:

\vskip1mm

a) $r_i$ is strictly cyclically reduced; \hspace{2mm}
b) $|r_i|=|f|$;  \hspace{2mm}
c) $|r_i|$  is even; \hspace{2mm}
d) $| \psi_i r_i \psi_i^{-1} | $ is even.

\vskip1mm 

If any one of these assertions is satisfied, we necessarily have
$\psi_i=\id$ (since the expression $\psi_i r_i \psi_i^{-1}$
is reduced).\\

Let us now explain the construction of $M$, that we perform
in several steps: \\

\underline{Step 1.} We associate a diagram to each $\psi_i
r_i \psi^{-1}_i$.

Our construction will involve a base point $O$ which will be
considered as a primary vertex.
Let $r_i=x_1 \ldots x_{m}$ be a normal form of $r_i$.

\vskip1mm

$\bullet$ Assume that $r_i$ is strictly cyclically reduced,
i.e. $m= |f|$.

The  diagram for $\psi_i r_i \psi_i^{-1} = r_i$ is the loop
at the base point $O$
consisting of $2m$ half-segments $d_1,d'_1, \ldots, d_m,
d'_m$ such that $\phi (d_jd'_j)=x_j$ for each $j$.

\vskip1mm

$\bullet$ Assume that $r_i$ is not strictly cyclically
reduced, i.e. $m= |f|+1$. Note that in this case $(x_m x_1)x_2\cdots x_{m-1}$ is strictly cyclically reduced.

The  diagram for $\psi_i r_i \psi_i^{-1}$ is a loop at a
vertex $v$
joined to the base point $O$  by a path.

Let $\psi_i = z_1 \ldots z_k$ be a normal form of $\psi_i $.

The path $Ov$ consists of $2k$ half-segments
$e_1,e_1', \ldots, e_k, e'_k$ such that $\phi (e_j e'_j) =
z_j$
for each $j$ and an additional final half-segment $e$.

The loop at $v$ consists of $2m-2$ half-segments 
$b,d_2,d'_2, \ldots, d_{m-1}, d'_{m-1},c$
such that $\phi (d_jd'_j)=x_j$ for each $j$.

The three  half-segments $e,b,c$ which meet at the secondary
vertex $v$
 are labeled to satisfy the necessary (and compatible)
conditions
$ \phi (eb) = x_1$, $\phi (ce^{-1}) = x_m$ and $\phi (cb) =
x_mx_1$.
For instance we can take $\phi(b) = x_1$, $\phi(c) = x_m$
and $\phi(e) = \id$.\\

\underline{Step 2.}
The  initial diagram for the composition
$$g=  (\psi _1 r_1 \psi _1 ^{-1}) \cdots (\psi _n r_n \psi
_n ^{-1})$$ consists of the initial diagrams 
for each $\psi _i r_i \psi _i ^{-1}$ arranged, in
counterclockwise order, around the base point $O$.
This initial diagram has the desired properties
(\ref{point(i)}) and (\ref{point(iii)}).\\

\underline{Step 3.}
We will now proceed to the identification of some
half-segments  of $M$ until
the boundary length of $M$ is $\leq 2$.

Note that in these identifications:
\begin{itemize}
 \item We shall always identify primary vertices with
primary vertices
and secondary vertices with secondary vertices, preserving
this distinction;
 \item The label of a segment will never be in the
amalgamated part;
 \item The number $n$ of regions of $M$ will not change
and (\ref{point(i)}) and (\ref{point(iii)}) will be
satisfied at each stage;
 \item If $\alpha$ is a boundary cycle of $M$, then $\phi
(\alpha)$
is conjugate to $g$.
\end{itemize}

For grounds of brevity, the tiresome and easy verification of the second point
(on label of segments) has been omitted in the two cases below.

If the boundary length of $M$ is $ \geq 3$, there
necessarily exists successive segments $ee'$ and $ff'$ in $\partial M$
such that the labels $ \phi (ee')$  and $\phi (ff')$ 
are in the same factor of Aut$[\C^2]$.
Indeed, otherwise, any boundary cycle
$\alpha= e_1 e_1'\ldots e_i e_i'$ of $M$
would have even length $i \geq 4$ and its label $\phi
(\alpha) =\phi (e_1e'_1) \ldots  \phi (e_ie'_i)$
would be a strictly cyclically reduced  conjugate of $g$: A
contradiction.\\

So we consider the element $s = \phi (ee') \phi (ff')$ which
lies in a factor of Aut$[\C^2]$.

\begin{description}
 \item[Case 1] Assume that $s$ is not in the amalgamated
part.\\

Change the label on the half-segment $e'$ to $1$,
readjusting the labels on the other half-segments
at the secondary vertex separating $e$ and $e'$.
In other words this amounts, for each half-segment $g$
ending at this secondary vertex,
to replace its label $\phi (g)$ by   $ \phi (ge')$.

In the same way, change the label on the half-segment $f$ to
$1$,
readjusting the labels on the other half-segments
at the secondary vertex separating $f$ and $f'$.

Then we identify the (oriented) half-segments $e'$ and
$f^{-1}$
(which now have the same labels) (see Fig.
\ref{fig:cas1}, where the $\bullet$ are primary vertices and the $\circ$ are  secondary vertices).\\

\begin{figure}[ht]
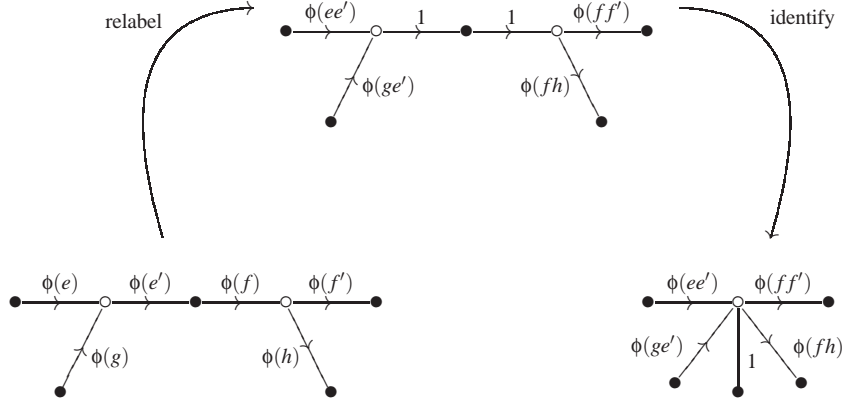

$$\dessinCasUn$$ 
\caption{Relabellings and identifications in case 1.}
\label{fig:cas1}
\end{figure}

 \item[Case 2] Assume that $s$ is in the amalgamated part.\\

Note first that the diagram has no loop of length
$\leq 2$ with total label in one of the factors of Aut$[\C^2]$.
Indeed, such a loop $\alpha$ would be a boundary cycle of
some
strictly smaller subdomain,
so that, by Lemma \ref{normalsubgroup} below, $\phi ( \alpha
)$
would be the product of strictly less that $n$ conjugates of
$f$.
This would contradict the minimality of $n$.

Therefore, if $u$ is the initial vertex of  $ee'$,
$v$ its  terminal  vertex (as well as the initial vertex of
$ff'$)
and $w$ the terminal vertex of $ff'$, then the vertices
$u,v,w$ are distinct.

Recall that $\phi(f)\phi(f') = \phi(e')^{-1}\phi(e)^{-1} s$.
We change the labels in the following way (see Fig.
\ref{fig:cas2}):
\begin{itemize}
 \item we change the label of $f$  to
$\phi (e')^{-1}$, readjusting the labels on the other
half-segments
at the secondary vertex separating $f$ and $f'$;
 \item we change the label of $f'$ to $\phi (e)^{-1}$;
 \item for each half-segment $g$ having $w$ as initial
vertex, we
replace its label $\phi (g)$ by   $s \phi (g)$.
\end{itemize}
Then we identify the (oriented) segments $e,e'$ and
$f'^{-1},f^{-1}$ (which now have the same labels).
\end{description}

\begin{figure}[ht]
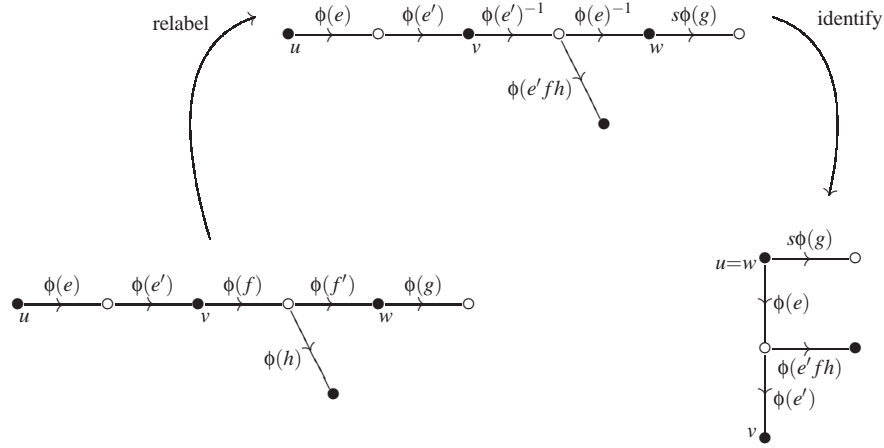

$$\dessinCasDeux$$ 
\caption{Relabellings and identifications in case 2.}
\label{fig:cas2}
\end{figure}

Note that after performing the identification in case 1
(resp. in case 2) the boundary length drops by 1 (resp. by
2).
Note also that if two regions $D_1$ and $D_2$ share at least
one half-segment,
and if $r_1, r_2$ are two boundary cycles of these regions
with respect to a common starting point,
then we can not have $r_1 = r_2^{-1}$.
Indeed, if it was the case, removing the two regions from
the diagram and applying Lemma \ref{normalsubgroup}
we would obtain a new element in $R(f)$ that would
contradicts the minimality of $n$. 
In fact, by Lemma \ref{intersectionoftwogeodesics}, two
regions in the diagram never share an edge of length greater
than 4.\\

\underline{Step 4.}
By induction, the previous step gave us a diagram with a
boundary length $\le 2$.
We now perform one last more identification to
obtain that the boundary length of $M$ is at most 1.
If the last identification falls under case 1 there is no
particular problem.
However, if we are in case 2, then we can no longer  assume
that the vertices $u$ and $w$ are disjoint.
So we slightly modify the procedure: we keep the label of
$f'$ to be $\phi(e)^{-1} s$
and we only identify the half-segments $e'$ and $f$.
It may happen that after this identification
the label of the segment $ef'$ on the boundary of $M$ is in
the amalgamated part:
apart from being slightly non aesthetic, this will not be
a problem in the proof of Theorem \ref{thm:B}.  
\end{proof}

\begin{lem} \label{reducedform}
Any conjugate of $f$ (notation as in Theorem
\ref{constructionofM}) can be written under reduced form 
$\psi r \psi^{-1}$,
where $r$ is a weakly cyclically reduced conjugate of $f$.
\end{lem}

\begin{proof}
Recall that a hyperbolic element of Aut$[\C^2]$
is strictly (resp. weakly) cyclically reduced if and only if
its geodesic contains (resp. intersects) the edge $e=\id (A \cap E)$
in the Bass-Serre tree (see Lemma \ref{strictlyandweaklycyclicallyreducedelements}).
Let now $g$ be a conjugate of $f$.
If the geodesic of $g$ intersects $e$,
we can just set $\psi=\id$, $r=g$.
Therefore, let us assume that this geodesic 
does not  intersect  $e$.

Let $\dist$ be the natural distance on the Bass-Serre tree
and  $I$ be the middle of the edge $e$.
For any element $h$ of $G$,
we have $|h| = \dist(I, h(I))$.

Let $p \in {\rm Geo}(g)$ be the unique vertex such that
$\dist( {\rm Geo}(g), e) = \dist(p,e)$.
Since $\dist(p,e) \geq 1$, there exists a unique point $I'$
on the geodesic $[p,I]$ such that $\dist(p,I')= \frac12$.
The group $G$ acting transitively on the middles of the
edges
of the Bass-Serre tree, there exists an element $\psi$ of
$G$
such that $\psi (I) = I'$.
Let us  set $r= \psi ^{-1} g \psi$.
We have ${\rm Geo}(r)= \psi ^{-1} ( {\rm Geo}(g))$
and $d( {\rm Geo}(g), I') =\frac12$,
so that $\dist( {\rm Geo}(r), I) = \frac12$
and $ {\rm Geo}(r)$ meets $e$, i.e.
$r$ is weakly cyclically reduced.
Finally, we have
$|g|= \dist(I,g(I))= \lon (g) + 2 \dist(I, {\rm Geo}(g) ) =|f| + 2 \dist(I,I') + 1$,
$| \psi | = \dist(I,I')$ and $|r| = |f| + 1$,
so that $|g| = | \psi | + | r| + | \psi ^{-1} |$.  \\
\end{proof}

The following result can be proven similarly as Lemma 1.2 in
\cite[p. 239]{LS} (that is, by induction on the number $m$
of regions).

\begin{lem} \label{normalsubgroup}
Let $M$ be an oriented connected and simply
connected diagram with $m$ regions
$D_1, \ldots,D_m$.
Let $\alpha$ be a boundary cycle of $M$
(beginning at some vertex  of $\partial M$) and
let $\beta_i$ be a boundary cycle of $D_i$
(beginning at some vertex  of $\partial D _i$), for $1 \leq
i \leq m$.
Then $\phi ( \alpha )$ belongs to the normal subgroup
generated
by the $\phi ( \beta _i )$, $1 \leq i \leq m$.
More precisely, there exists $u_1, \ldots,u_m$ in
Aut$[\C^2]$
such that

\vskip1mm

\centerline{ $\phi ( \alpha ) =(u_1 \, \phi (\beta _1) \,
u_1^{-1}) \ldots (u_m \,  \phi (\beta _m) \, u_m^{-1})$. }
\end{lem}

\vskip5mm

\subsection{A  dictionary between  Bass-Serre and
Lyndon-Schupp theories}

Let $\alpha$ be a boundary cycle of some region of $M$ (as in Theorem
\ref{constructionofM})
beginning at some vertex $v$.
If $v$ is primary (resp. secondary), $\phi ( \alpha)$ is a
reduced form
of a strictly cyclically reduced (resp. non strictly
cyclically reduced)
element of $R(f)$.

\begin{lem} \label{intersectionoftwogeodesics}
If $D_1,D_2$ are two distinct regions of a diagram $M$ 
having a common edge,
there exists a primary vertex $v$ of $\partial D_1 \cap
\partial D_2$ such that
the labels $g_1,g_2$  of the boundary cycles of $D_1,D_2$
beginning at $v$ satisfy
\vskip1mm
\centerline{$| {\rm Geo}(g_1) \cap {\rm Geo}(g_2)| \, \geq
\, | \partial D_1 \cap \partial D_2 | $.}
\end{lem}

\begin{proof}
If $k$ is the largest integer such that $k < | \partial D_1
\cap \partial D_2 | $,
there exists a path of $k$ segments $s_1,\ldots,s_k$
included into
$\partial D_1 \cap \partial D_2$.
We can just take for $v$ the initial or terminal vertex of
this path
(if $k=0$, these two vertices  coincide).
Indeed, we may assume that $g_1$ has the normal form 
$g_1= \phi (s_1) \ldots \phi (s_k) x_1 \ldots x_m$
(where each $x_i$ is in some factor of $G$).
Therefore,  $g_2^{-1}$ has the normal form
$g_2^{-1}= \phi (s_1) \ldots \phi (s_k) y_1 \ldots y_m$
(where each $y_i$ is in some factor of $G$).

The geodesics  ${\rm Geo}(g_1)$ and  ${\rm
Geo}(g_2^{-1})={\rm Geo}(g_2)$ contain the $k+1$ consecutive
edges:

$$\id \, (A \cap E),\: \phi (s_1) \,  (A \cap
E), \ldots,
\phi (s_1) \ldots \phi (s_k) \, (A \cap E). $$ 
\end{proof}

\begin{exple}
Assume that $M$ contains the two regions depicted in Fig. \ref{fig:expleLS}
(the $\bullet$ are primary vertices, the secondary vertices
are denoted by $\circ$ only when they have valence $\ge
3$).

\begin{figure}[h]
$$
\xygraph{
!{<0cm,0cm>;<1cm,0cm>:<0cm,1cm>::}
!{(-1,0)}*{\bullet} ="0l"
!{(1,0)}*{\bullet} ="0r"
!{(-2,0)}*{\circ} ="l"
!{(2,0)}*{\circ} ="r"
!{(-3,1)}*{\bullet} ="ul"
!{(3,1)}*{\bullet} ="ur"
!{(-3,-1)}*{\bullet} ="dl"
!{(3,-1)}*{\bullet} ="dr"
!{(-.5,-.6)}*{D_1}
!{(-.5,.6)}*{D_2}
"l"-|*@{<}@{-}^{e_4}"ul"
"ul"-|*@{<}@{-}^{a_3}"ur"
"ur"-|*@{<}@{-}^{e_3}"r"
"l"-|*@{>}@{-}_{e_1}"dl"
"dl"-|*@{>}@{-}_{a_2}"dr"
"dr"-|*@{>}@{-}_{e_2}"r"
"0l"-@{-}_{\id }"l"
"0r"-@{-}^{\id }"r"
"0l"-|*@{<}@{-}^{a_1}^>{v}"0r"
}
$$
\caption{} \label{fig:expleLS}
\end{figure}

We get $g_1 = a_1e_1a_2e_2$, $g_2 = e_3a_3e_4a_1^{-1}$
and Fig. \ref{fig:expleBS} gives the picture in the Bass-Serre tree. 
Note that here for simplicity we took $D_1$ and $D_2$ with boundary length $4$,
but in the context of Theorem \ref{constructionofM} any region has boundary length at least $10$.

\begin{figure}[h]
$$\xygraph{
!{<0cm,0cm>;<1cm,0cm>:<0cm,1cm>::}
!{(0,2)}*{\bullet} ="e2A"
!{(0,0)}*{\bullet} ="e3A" 
!{(1,1)}*{\bullet} ="idE"
!{(2.5,1)}*{\bullet} ="idA"
!{(4,1)}*{\bullet} ="a1E"
!{(5,2)}*{\bullet} ="e1A"
!{(5,0)}*{\bullet} ="e4A" 
!{(1,2)}="g1i"
!{(4,2)}="g1f"
!{(4,0)}="g2i"
!{(1,0)}="g2f"
"e2A"-_<{e_2^{-1}A}"idE"
"e3A"-^<{e_3A}"idE"
"idA"-_>(0.9){\id E}_<{ \id A}"idE"
"e1A"-_<{a_1e_1A}"a1E"
"e4A"-^<{a_1e_4^{-1}A}"a1E"
"idA"-^>(0.9){a_1E}"a1E"
"g1i"-@{.>}@/_0.5cm/^{g_1}"g1f"
"g2i"-@{.>}@/_0.5cm/^{g_2}"g2f"
}$$
\caption{} \label{fig:expleBS}
\end{figure}

\end{exple}

\begin{lem} \label{lem:valence3}
If $v$ is a vertex of valence $3$ of $M$ 
with regions $D_1, D_2, D_3$ meeting at $v$
and if $g_1,g_2,g_3$ are the labels of the boundary cycles
of these regions beginning at $v$,
then the geodesics of the $g_i$'s form a tripod in the
Bass-Serre tree and for all $i,j$'s:

\vskip1mm

\centerline{$| {\rm Geo}(g_i) \cap {\rm Geo}(g_j)) | \, 
\geq  \, | \partial D_i \cap \partial D_j |$.}

\end{lem}

\begin{proof}
The vertex $v$ is necessarily secondary.
Let $e_1$ (resp. $e_2$, resp. $e_3$) be the (oriented)
half-segment
having $v$ as initial vertex and included into $\partial D_2
\cap \partial D_3$
(resp. $\partial D_1 \cap \partial D_3$, resp. $\partial D_1
\cap \partial D_2$).
The $\phi (e_i)$'s  are in the same factor of $G$ and if $i
\neq j$,
$\phi (e_i)\phi (e_j)^{-1}$ is not in the amalgamated part.
As in Lemma \ref{intersectionoftwogeodesics}, let $k$ be the
largest integer such that
$k < | \partial D_1 \cap \partial D_2 | $ and let $s_1,
\ldots, s_k$ be  segments such that
the path $e_3,s_1,\ldots,s_k$ is included into $\partial D_1
\cap \partial D_2$.
We may assume that $g_1$ has the normal form 

\vskip1mm

\centerline{$g_1= \phi (e_3) \phi (s_1) \ldots \phi (s_k)
x_1 \ldots x_m \phi (e_2)^{-1}$,}

\vskip1mm

\noindent where each $x_i$ is in some factor of $G$.
Therefore, $g'_1= \phi (e_3)^{-1} g_1 \phi (e_3)$ is
strictly cyclically reduced and has the normal form

\vskip1mm

\centerline{$g'_1=  \phi (s_1) \ldots \phi (s_k) x_1 \ldots 
x_{m+1}$,}

\vskip1mm

\noindent where $x_{m+1}=\phi (e_2)^{-1} \phi (e_3)$.
Since the geodesic of $g'_1$ contains the consecutive edges

\vskip1mm

\centerline{$\id \, (A \cap E),\: \phi (s_1) \,  (A \cap E),
\ldots, \phi (s_1) \ldots \phi (s_k) \, (A \cap E)$,} 

\vskip1mm

\noindent it is clear that the geodesic of $g_1$ contains
the consecutive edges

\vskip1mm

\centerline{$\phi(e_3)  \, (A \cap E),\: \phi(e_3) \phi
(s_1) \,  (A \cap E), \ldots,
\phi(e_3) \phi (s_1) \ldots \phi (s_k) \, (A \cap E)$.}

\vskip1mm

One would show in the same way that these edges are also
contained in the geodesic of $g_2$,
so that we get $| {\rm Geo}(g_1) \cap {\rm Geo}(g_2)| \,
\geq \, | \partial D_1 \cap \partial D_2 | $.
The other inequalities are proven similarly.
We finish the proof by noting that
$ {\rm Geo}(g_1) \cap {\rm Geo}(g_3)$ contains the edge
$\phi(e_2)  \, (A \cap E)$
and that $ {\rm Geo}(g_2) \cap {\rm Geo}(g_3)$ contains the
edge $\phi(e_1)  \, (A \cap E)$.
If the $\phi (e_i)$'s  are in the factor $A$ (resp. $E$),
it is clear that the three edges $\phi (e_i) \, (A \cap E)$
intersect at the vertex $\id A$
(resp. $\id E$).  
\end{proof}

\section{The proof of Theorem \ref{thm2}} \label{sec:proofoftheorem3}

\subsection{A result about curvature}
 
Let us recall some notations from \cite{LS}.
If $v$ is a vertex of a diagram $M$,  the degree $d(v)$ (or
valence) of $v$ will denote the number of oriented edges
having $v$ as initial vertex
(thus, if an edge has both endpoints at $v$, we count it
twice).
If $D$ is a region, the degree $d(D)$ of $D$ will denote the
number of edges of $D$.
% (i.e. the number of faces $D'$ such that $D \cap D'$ is
% an edge).
% la parenthese n'est PAS CORRECTE s'il y a plusieurs
% composantes connexes dans l'intersection...
The following formula defines a curvature contribution for
each region:
$$\delta(D) = 2 - d(D) + \sum_{v\in D} \frac{2}{d(v)}.$$

\begin{lem}\label{lem:courbure}
For any diagram on the 2-sphere,
we have $$ 4 = \sum_D \delta(D).$$
\end{lem}

\begin{proof} 
Let $V,E,F$ be the numbers of vertices, edges and faces of
the diagram.
The formula is a direct consequence of Euler's formula on
the sphere
$2 = V - E +F$ and of the obvious relations $2E =
\sum_{(v,D)} 1$,
$V = \sum_{(v,D)} \frac{1}{d(v)}$ and  $F=\sum_{(v,D)}
\frac{1}{d(D)}$:
$$4 =  2V + 2F - 2E =  \sum_{(v,D)} \left( \frac{2}{d(v)} +
\frac{2}{d(D)} -1 \right) = \sum_D \delta(D) $$
where the first sum runs over the couples $(v,D)$  with $v$
a vertex and $D$  a face such that $v\in D$.  
\end{proof}

\begin{cor}\label{disk}
For any planar diagram homeomorphic to the disk,
we have 
$$ 2 \leq \sum_D \delta(D).$$
\end{cor}

\begin{proof}
Let $K$ be this diagram.
Let $L$ be the spherical diagram obtained by sticking along
their boundaries
two copies $K_1$ and $K_2$ of $K$.
Since $L$ is homeomorphic to the sphere,
we have $\displaystyle{ 4 = \sum_{D\in L} \delta(D)}$,
i.e. $$ 4 = \sum_{D\in K_1} \delta(D) +\sum_{D\in K_2} \delta(D)=
 2 \sum_{D\in K_1} \delta(D) \leq 2 \sum_{D \in K} \delta(D).$$
The last inequality comes from the fact that 
for each boundary region $D$ in $K$ the contribution
curvature $\delta(D)$
computed in the disk diagram is bigger than the
contribution computed in the spherical diagram.
\end{proof}

\begin{rem}\label{rem:neg}
Here is a (non complete) list of faces $D$ having negative
or zero curvature:
\begin{itemize}
\item $D$ with $d(D) \ge 6$;
\item $D$ with $d(D) =5$ and at most 3 vertices of $D$ are
tripods;
\item $D$ with $d(D) = 4$ and each vertex of $D$ has valence
at least 4;
\item $D$ with $d(D) = 4$ and $D$ admits a tripod, two
vertices of  valence at least 4
and a fourth vertex of valence at least 6;
\item $D$ with $d(D) = 3$ and each vertex of $D$ has valence
at least 6.
\end{itemize}
\end{rem}

\subsection{The end of the proof}

We are now in position to prove Theorem \ref{thm2}.
As in Theorem \ref{thm1}, we will prove a stronger and more
geometric version:

\begin{thm}\label{thm:B}
If $f \in G$ is a hyperbolic element of geometric length
$\lon(f) \ge 14$ satisfying condition $(C2)$,
then the normal subgroup generated by $f$ in Aut$[ \C^2 ]$
is different from $G$.
\end{thm}

\begin{proof}
We can assume that $f$ is a strictly cyclically reduced
element of length $ \lon(f) = |f| = 2l \geq 14$.
If the normal subgroup generated by $f$ in  Aut$[ \C^2 ]$
was equal to $G$ then, 
by Theorem \ref{constructionofM},
there would exist an ${\rm Aut}[\C^2]$-labeled oriented diagram $M$ such
that:

\begin{enumerate}
 \item $M$ is connected and simply connected;
\item  The perimeter of $M$ is $\leq 1$;
\item If $e_1 e_1' \ldots e_te_t'$ is
a boundary cycle of some region of $M$,
then $t= |f|$ and $\phi (e_1e_1') \ldots \phi (e_te_t')$ is
a reduced form
of a strictly cyclically reduced conjugate of  $f$.
\end{enumerate}

Let $D_1,D_2$ be two distinct regions of $M$ having a common
edge. By Proposition \ref{prop:geo}
and Lemma \ref{intersectionoftwogeodesics}, we have $|
\partial
D_1 \cap \partial D_2 | \leq  4$.
Since $ | \partial D_1 | \geq 14$, we conclude that any
interior region has at least 4 edges.

Furthermore, if $D_1,D_2,D_3$ are three distinct regions of
$M$ having a common vertex of valence 3,
by Lemmas \ref{lem:tripod} and \ref{lem:valence3},
we know that each edge $\partial D_i \cap \partial D_j$ is
at most of length 2.
In consequence, if an interior region has at least 1
interior vertex of valence 3,
then this region has at least 5 edges.
Similarly, if  an interior region has at least 3 interior
vertices of valence 3,
then this region has at least 6 edges. 

By the previous observations, and using Remark
\ref{rem:neg},
we conclude that the curvature contribution $\delta(D)$ of
any interior region $D$ is non positive.
Let us examine now the contribution of the boundary regions.
Since the perimeter is at most 1
(i.e. at most two half-segments), there are at most 2
boundary regions. 

Suppose there are exactly 2 boundary regions. Since the
boundary edge of such a region $D$ is an half-segment,
it is easy to check that $D$ has at least 5 edges,
and that if at least one interior vertex is of valence 3
then $D$ has at least   6 edges.
Thus $\delta(D) \le 0$.

Assume now that there is only 1  boundary region $D$. Then
the only boundary vertex of $D$
(which has to be counted twice) has valence at least 4. So
$D$ has at least $5$ edges,
and if $D$ has exactly $5$ edges then the 3 interior
vertices can not be of valence 3,
and again we obtain $\delta(D) \le 0$.

In conclusion we have $\sum \delta(D) \le 0$, which is
contradictory with Lemma \ref{disk}.
We conclude that the normal subgroup generated by $f$ in Aut$[ \C^2 ]$ can not be equal to
$G$. 
\end{proof}

\section{The remaining cases: length 10 and
12}\label{sec:10-12}

In this section we present some of the problems that await
 the reader who would like to extend our results to the case
of an automorphism of length 10 or 12, along with two
striking examples of configuration in the Bass-Serre tree.
 
\subsection{Length 12}

The main problem to adapt our strategy to the case of $f$
with $\lon(f) =12$ is
that we have to deal with regions in a $R(f)$-diagram that
are triangles with 3 edges of length 4.
Then we would have to study not only tripods  coming from 3 conjugates of $f$,
but their generalization, which we call $n$-pods, coming from
$n$ conjugates $f_i$ ($0 \leq i \leq n-1$) of $f$.
It is the case where the geodesics $\geo(f_i)$ have a common vertex
and where each pair $\geo(f_i), \geo(f_{i+1})$ has at least one
edge in common (where $i = 0, ..., n-1$
and the index are taken modulo $n$). Precisely to be sure
that the curvature of such a triangle is non positive
it would be sufficient to have the following 

\begin{lemconj}\label{lem:conj}
 If $n$ conjugates of $f$ form a $n$-pod in the Bass-Serre
tree, with two consecutive branches of length 4, then $n \ge
6$. 
\end{lemconj}

We believe that this result is true, but the verification
seems to have to involve a very long list of cases:
that is why we do not think reasonable to try to present a
proof.
However it is interesting to note that there exists $6$-pods
with branches of length 4.

\begin{exple}[6-pod with all branches of length
4]\label{exple:6pod}

\begin{figure}[ht]
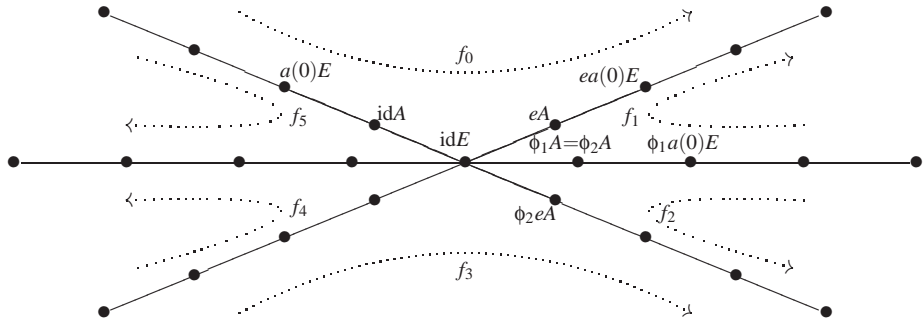

$$\xygraph{
!{<0cm,0cm>;<1.5cm,0cm>:<0cm,1cm>::}
!{(-3.2,2)}*{\bullet} ="h-4" !{(-2.4,1.5)}*{\bullet} ="h-3"
!{(-1.6,1)}*{\bullet} ="h-2" !{(-.8,.5)}*{\bullet} ="h-1"
!{(+3.2,2)}*{\bullet} ="h+4" !{(+2.4,1.5)}*{\bullet} ="h+3"
!{(+1.6,1)}*{\bullet} ="h+2" !{(+.8,.5)}*{\bullet} ="h+1"
!{(-3.2,-2)}*{\bullet} ="b-4" !{(-2.4,-1.5)}*{\bullet}
="b-3"
!{(-1.6,-1)}*{\bullet} ="b-2" !{(-.8,-.5)}*{\bullet} ="b-1"
!{(+3.2,-2)}*{\bullet} ="b+4" !{(+2.4,-1.5)}*{\bullet}
="b+3"
!{(+1.6,-1)}*{\bullet} ="b+2" !{(+.8,-.5)}*{\bullet} ="b+1"
!{(-4,0)}*{\bullet} ="-4" !{(-3,0)}*{\bullet} ="-3"
!{(-2,0)}*{\bullet} ="-2" !{(-1,0)}*{\bullet} ="-1"
!{(0,0)}*{\bullet} ="0" 
!{(+4,0)}*{\bullet} ="+4" !{(+3,0)}*{\bullet} ="+3"
!{(+2,0)}*{\bullet} ="+2" !{(+1,0)}*{\bullet} ="+1"
!{(-2,2)} ="f0i" !{(2,2)} ="f0f"
!{(3,.5)} ="f1i" !{(2.9,1.4)} ="f1f"
!{(3,-.5)} ="f2i" !{(2.9,-1.4)} ="f2f"
!{(-2,-2)} ="f3i" !{(2,-2)} ="f3f"
!{(-3,-.5)} ="f4i" !{(-2.9,-1.4)} ="f4f"
!{(-3,.5)} ="f5i" !{(-2.9,1.4)} ="f5f"
"h-4"-"b+4" "b-4"-"h+4" "-4"-"+4"
"h-2"-^<{a(0) E}"h-1" "h+1"-^>(0.85){e a(0) E}"h+2"
"h-1"-^<{\id A}^>(0.85){\id E}"0"
"0"-_>{\phi_2eA}"b+1" "0"-^>{eA}"h+1"
"0"-^>{\phi_1 A = \phi_2 A}"+1"
"+1"-^>{\phi_1 a(0)E}"+2" 
"f0i"-@{.>}@/_0.8cm/^{f_0}"f0f"
"f1i"-@{.>}@/^2cm/^{f_1}"f1f"
"f2i"-@{.>}@/_2cm/^{f_2}"f2f"
"f3i"-@{.>}@/^0.8cm/_{f_3}"f3f"
"f4f"-@{.>}@/_2cm/_{f_4}"f4i"
"f5f"-@{.>}@/^2cm/^{f_5}"f5i"
}$$
\caption{A 6-pod with all branches of length 4 (example
\ref{exple:6pod})} 
\label{fig:6pod}
\end{figure}

Let us consider the following automorphism $f_0$ of length $2l \ge 8$:
$$f_0 = e_1 a e_2 a \cdots e_l a$$
where $a = a(0) = (y,-x)$. We suppose that $ e_1 = (x+P(y),y)$, and we set $e =e_1$. 
We are going to construct $f_1, \cdots, f_5$ five conjugates
of $f_0$
such that  their geodesics form a 6-pod (see Figure
\ref{fig:6pod}).

For $i = 1, \cdots, 5$ we choose constants $c_i \neq 0$ and we
set $t_i = (x,y+c_i)$. We take $f_i = \phi_i f_0
\phi_i^{-1}$ where
\begin{eqnarray*}
\phi_1 & = & e t_1 e^{-1}  \\
\phi_2 & = & e t_1  e^{-1} t_2 \\
\phi_3 & = & e t_1  e^{-1} t_2 e t_3 e^{-1} \\
\phi_4 & = & e t_1  e^{-1} t_2 e t_3 e^{-1} t_4 \\
\phi_5 & = & e t_1  e^{-1} t_2 e t_3 e^{-1} t_4 e t_5 e^{-1}
\end{eqnarray*}
are all elements of $E$. 

We claim that for each $i = 0, \cdots, 4$,  the geodesics of $f_i$ and $f_{i+1}$ share a path of 4 edges with $\id E$ as an extremity.

Consider the case $i = 0$. We have $\geo(f_1) = \phi_1(\geo(f_0))$. Recall that $t_1$ fixes the ball of radius $2$ centered on $a(0)E$ (Remark \ref{subtreefixedbyatranslationbis}), so $\phi_1$ fixes the ball of radius $2$ centered on $ea(0)E$, hence the claim.

Now take $i = 1$. Note that $f_2 = \phi_1 t_2 f_0 t_2^{-1} \phi_1^{-1} =
\phi_1 t_2 \phi_1^{-1} f_1  \phi_1 t_2^{-1} \phi_1^{-1}$, and $\phi_1 t_2 \phi_1^{-1}$ fixes the ball of radius 2 centered at $\phi_1 a(0) E$. 
Thus the geodesic of $f_1$ and $f_2$ share 4 edges.
We can make a similar computation for $i = 2,3,4$.  
 
Suppose now that the constants $c_i$ satisfy:
\begin{eqnarray*}
c_1 +  c_2 + c_3 &=& 0 \\
c_2 +  c_3 + c_4 &=& 0 \\
c_3 +  c_4 + c_5 &=& 0 
\end{eqnarray*}
For instance one can take $(c_1, c_2, c_3 , c_4 , c_5) =
(1,1,-2,1,1)$.

A straightforward computation shows that
 $$\phi_5 = et_1e^{-1}t_2et_3e^{-1}t_4et_5e^{-1}= $$
$$ (x  + P(y + c_1 + c_2 + c_3 + c_4 + c_5 ) - 
 P(y +  c_2 + c_3 + c_4 + c_5 ) +  P(y + c_3 + c_4 + c_5 )
$$
$$ -  P(y + c_4 + c_5 )
+P(y +  c_5 ) - P(y ), y +c_1 + c_2 + c_3 + c_4 + c_5  ) =
(x, y - c_3).$$
Since $(x,y-c_3)$ fixes the ball of radius 2 centered at $a(0)E$, this implies that the geodesics of $f_0$ and $f_5$ share
4 edges, as shown on Fig. \ref{fig:6pod}. 
\end{exple}

\subsection{Length 10}

The case of $f$ of length $10$ seems even more doubtful.
For instance one could have pentagonal regions with all
edges of length $2$ and all vertices of valence $3$.
It is probably easy to rule out this case, but there are
some harder ones.
One could have triangular regions with edges of length
$4,4,2$.
The  example \ref{exple:6pod} allows us to glue $6$ such
triangles along their edge of length $4$,
to obtain a $R(f)$-diagram with boundary length $12$.
One can wonder if it possible to glue two such diagrams to
obtain a $R(f)$-diagram on a sphere
(in this case our strategy would fail). One would need to
have 4-pods with branches $4,2,4,2$.
We do not know if this is possible,
but the following example shows  again that we would have to
rely on very careful computations
to exclude this case (note also that the
assumption 'consecutive' was crucial in the statement of
Lemma \ref{lem:conj})   

\begin{exple}[4-pod with branches of length 4, 1, 4,
1]\label{exple:4pod}

Similarly to the previous example we take $f_i = \phi_i f_0
\phi_i^{-1}$ where
\begin{eqnarray*}
\phi_1 & = & e t_1 e^{-1}  \\
\phi_2 & = & e t_1  e^{-1} t_2 \\
\phi_3 & = & e t_1  e^{-1} t_2 e t_3 e^{-1} \\
\end{eqnarray*}
with $t_1 = t_3 = (x+c,y)$ and $t_2 = (-x, y -c)$. Then one
can verify that 
$\phi_3 = (-x, y+c)$ and the geodesics of the $f_i$ form a
4-pod as on Figure \ref{fig:4pod}.
\end{exple}

\begin{figure}[ht]
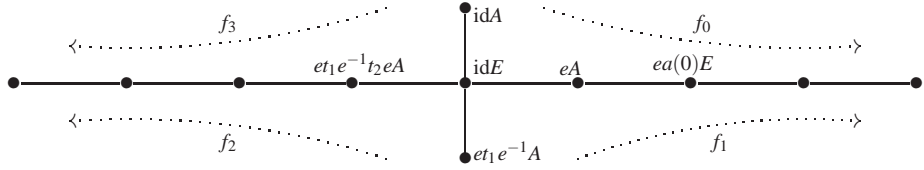

$$\xygraph{
!{<0cm,0cm>;<1.5cm,0cm>:<0cm,1cm>::}
!{(0,1)}*{\bullet} ="haut" !{(0,-1)}*{\bullet} ="bas"
!{(-4,0)}*{\bullet} ="-4" !{(-3,0)}*{\bullet} ="-3"
!{(-2,0)}*{\bullet} ="-2" !{(-1,0)}*{\bullet} ="-1"
!{(0,0)}*{\bullet} ="0" 
!{(+4,0)}*{\bullet} ="+4" !{(+3,0)}*{\bullet} ="+3"
!{(+2,0)}*{\bullet} ="+2" !{(+1,0)}*{\bullet} ="+1"
!{(0.7,1)} ="f0i" !{(3.5,.5)} ="f0f"
!{(1,-1)} ="f1i" !{(3.5,-.5)} ="f1f"
!{(-0.7,-1)} ="f2i" !{(-3.5,-.5)} ="f2f"
!{(-0.7,1)} ="f3i" !{(-3.5,.5)} ="f3f"
"-4"-"-3" "0"-^>{et_1e^{-1}A}"bas" "haut"-^<{\id
A}^>(0.85){\id E}"0"
"-3"-"-2" "-2"-"-1" "-1"-^<{et_1e^{-1}t_2eA}"0"
"0"-^>{eA}"+1"
"+1"-^>{ea(0)E}"+2" "+2"-"+3" "+3"-"+4"
"f0i"-@{.>}@/_0.2cm/^{f_0}"f0f"
"f1i"-@{.>}@/^0.2cm/_{f_1}"f1f"
"f2i"-@{.>}@/_0.2cm/^{f_2}"f2f"
"f3i"-@{.>}@/^0.2cm/_{f_3}"f3f"
}$$
\caption{A 4-pod with branches 4,1,4,1 (example
\ref{exple:4pod})} \label{fig:4pod}
\end{figure}

\section*{Annex: genericness of condition $(C2)$ } \label{annex}

We begin with a reformulation of Theorem \ref{thm:B}:

\begin{thm} \label{thestrongesttheorem}
Let $l \geq 7$ be an integer.
Assume that the polynomials $P_1,\ldots, P_l \in \C[y]$
are general and independent.
If the element $f$ of $G$ can be written $f=a_1e_1 \ldots a_le_l$
where $e_i=e(P_i)$ and
$a_i \in  A \setminus E$ for each $i$,
then the normal subgroup generated by $f$ in Aut$[\C^2]$
is different from $G$.
\end{thm}

In this section, we will show that if $P_1, \ldots, P_l$
are generic (in some sense), then they are general and independent.
We will also finish by giving explicit examples.

\subsection*{A. Genericness of condition $(C1)$}
\label{sec:annexA}

The aim of this subsection is to show that  condition $(C1)$
is generic (see Corollary \ref{cor1} and Remark \ref{rem1}).
For technical purposes we introduce a variation of the
notion of general polynomial (see Def. \ref{def:general}).

\begin{lem} \label{equivalentassertionsofgeneralpolynomials}
Let $Q \in \C[y]$ be a polynomial.
The following assertions are equivalent:

\begin{enumerate}
\item $\forall \, \alpha, \beta, \gamma \in \C, 
\hspace{3mm} 
Q(y) = \alpha Q ( \beta y + \gamma)
\hspace{3mm} \Longrightarrow \alpha = \beta =1 \mbox{ and }
\gamma=0;$
\item $\forall \, \alpha, \beta, \gamma \in \C, 
\hspace{3mm} 
 Q(y) = \alpha Q ( \beta y + \gamma)
\hspace{3mm} \Longrightarrow \beta =1.$
\end{enumerate}
\end{lem}

\begin{proof} (1) $\Rightarrow$ (2) is obvious.
Let us prove (2) $\Rightarrow$ (1).
If $Q$ satisfies (2), note that $Q$ can not be constant.
If $Q(y) = \alpha Q (  y + \gamma)$, 
it is enough to show that $\gamma = 0$.
Let $\zeta$ be a root of $Q$.
Since $\zeta + n \gamma$ is also a root of $Q$ for any
integer $n$, we must have $\gamma = 0$.  
\end{proof}

\begin{defi}
We say that $Q$ is \textbf{weakly general} if it satisfies
the equivalent assertions
of Lemma \ref{equivalentassertionsofgeneralpolynomials}.
\end{defi}

\begin{rem} 
\label{rem:general}
Clearly if $Q^{'}$ is weakly general then $Q$ is also weakly
general. 
Furthermore, $Q^{(k)}$ is weakly general if and only if
the following equivalent assertions are satisfied:
\begin{enumerate}
\item $\forall \, \alpha, \beta, \gamma \in \C, 
\hspace{3mm} 
\deg ( Q(y) - \alpha Q ( \beta y + \gamma)) < k
\hspace{3mm} \Longrightarrow \alpha = \beta =1 \mbox{ and }
\gamma=0$;
\item $\forall \, \alpha, \beta, \gamma \in \C, 
\hspace{3mm} 
\deg ( Q(y) - \alpha Q ( \beta y + \gamma)) < k
\hspace{3mm} \Longrightarrow \beta =1$.
\end{enumerate}
In other words, a polynomial $Q$ of degree $d \ge 5$
is general if and only if $Q^{(d-3)}$ is weakly general.
\end{rem}

\begin{lem} 
\label{lem:weaklygeneral}
The following assertions are equivalent:
\begin{enumerate}
\item $Q$ is not weakly general;
\item there exists $c \in \C$, $R \in \C [y]$,
$k \geq 0, n \geq 2$  such that $Q(y+c) =y^k
R(y^n)$.
\end{enumerate}
\end{lem}

\begin{proof} (1) $\Longrightarrow$ (2).
If $Q$ is not weakly general, there exists $\alpha, \beta,
\gamma$ with $\beta \neq 1$
such that $Q(y) = \alpha Q ( \beta y + \gamma)$.
If we set $c=\frac{\gamma}{1- \beta}$, then the polynomial
$P(y)=Q(y+c)$
satisfies $P(y) = \alpha P ( \beta  y)$.
Writing $P= \sum_i p_i y^i$, the last equation is equivalent
to $\forall i, (1- \alpha \beta ^i) p_i=0$.
If $\beta$ is not a root of unity, this implies that there
exists  $k \geq 0$ such that $P= p_k y^k$.
Assume now that $\beta$ is a primitive $n$th root of
the unity. If $P \neq 0$, there exists $k \geq 0$
such that $p_k \neq 0$ and so $\alpha = \beta^{-k}$.
Since $p_i \neq 0$ implies $i \equiv k \pmod n$,
we get $P= y^k R(y^n)$, where $R(y)= \sum_i p_{k+ni}y^i$.

(2) $\Longrightarrow$ (1).
This is a consequence of the previous computation. 
\end{proof}

\begin{prop}
\label{prop:gen1}
\begin{enumerate}
\item 
If $d \geq 3$, the generic element of $\C [y] _{\leq \, d}$
is weakly general;
\item 
If $d \geq 5$, the generic element of $\C [y] _{\leq \, d}$
is  general.
\end{enumerate}
\end{prop}

\begin{proof}
If $u \in \R$, we denote its integer part  by $[u]$.

(1) If $Q \in \C [y] _{\leq \, d}$ is not weakly general,
by Lemma \ref{lem:weaklygeneral}
we can write
$$Q(y)= (y-c)^k R \left( ( y-c)^n\right)$$
where $0 \leq k \leq d$,
$2 \leq n\leq d$, $c \in \C$, $e= [ d/n ]$ and
$R \in \C [y]_ {\leq \, e }$.
Therefore, $Q$ belongs to the image of the following morphism

$\varphi_{k,n}: \C \times \C[y]_{\leq \, e} \to \C[y]$,
$(c,R(y)) \mapsto (y-c)^k R ( ( y-c)^n)$.
However
$$\dim {\rm Im} \varphi_{k,n} \leq \dim (\C
\times \C[y]_{\leq \, e})=e+2
\leq \frac{d}{n}+2 \leq \frac{d}{2} +2 < d+1 =
\dim \C [y]_{\leq \, d} .$$

(2) is a direct consequence of (1), by considering the map
$Q \mapsto Q^{(d-3)}$, and using Remark \ref{rem:general}.  
\end{proof}

\begin{prop}
\label{prop:gen1bis}
If $d_1,d_2 \geq 5$ and $(P_1,P_2)$ is a generic element of
$\C [y]_{\leq \, d_1} \times \C [y]_{\leq \, d_2}$,
then $P_1$, $P_2$ represent different colors.
\end{prop}

\begin{proof}
By Lemma \ref{obviousresult}, if $P_1,P_2$ represent the same color, then
$(P_1,P_2)$ belongs to the image of the following morphism:
$\varphi : \C [y]_{\leq \, d_1} \times \C^5 \to \C [y] \times \C [y]$,
$(P_1, (\alpha, \beta,\gamma,\delta, \epsilon)) \mapsto
(P_1, \alpha P_1( \beta y + \gamma) + \delta y + \epsilon)$.
However, 
\vskip1mm
\hspace{25mm}{$\dim {\rm Im} \varphi \leq d_1 +6 <
\dim \C [y]_{\leq \, d_1} \times \C [y]_{\leq \, d_2}.$} \hfill $\square$
\end{proof}

\begin{rem}
If $d_1 \neq d_2$, Proposition \ref{prop:gen1bis} is still more obvious.
Indeed, the generic element $P_i$ of $\C [y]_{\leq \, d_i}$ has degree $d_i$.
Therefore, if $(P_1,P_2)$ is a generic element of
$\C [y]_{\leq \, d_1} \times \C [y]_{\leq \, d_2}$,
then $\deg P_1 \neq \deg P_2$,
which clearly implies that $P_1$, $P_2$ represent different colors.
\end{rem}

Propositions \ref{prop:gen1} and \ref{prop:gen1bis} give us the following result.

\begin{cor} \label{cor1}
Fix a sequence of integers $d_1,\ldots,d_l \geq 5$.
If  $(P_1, \cdots, P_l)$
is a generic element of $\prod_{1 \, \leq \, i \, \leq \, l} \C [y]_{\leq \, d_i}$,
then the polynomials $P_i$ are general and represent distinct colors.
\end{cor}

\begin{rem} \label{rem1}
In other words, if $a_i \in A \setminus E$ and $e_i =e(P_i)$ for $1 \leq i \leq l$,
then the automorphism $a_1 e_1 \ldots a_le_l$ satisfies condition $(C1)$.
\end{rem}

\subsection*{B. Genericness of condition $(C2)$}
\label{sec:annexB}

The aim of this subsection is to show that  condition $(C2)$
is generic (see Corollary \ref{cor3} and Remark \ref{rem3}).

\begin{prop}
\label{prop:gen2}
If $d_1,d_2,d_3 \geq 8$ and
$(P_1,P_2,P_3)$ is generic in
$\prod_{1 \, \leq \, i \, \leq \, 3} \C [y]_{\leq \, d_i}$,
then the polynomials $P_1,P_2,P_3$ are independent.
\end{prop}

\begin{proof}
By permutations, it is enough to show the  following two
points:

1) If $(P_1,P_2)$ is generic in $\C [y]_{\leq d_1} \times \C
[y]_{\leq d_2}$,
then $(A \cap E) e(P_2) (A \cap E)$ is not a mixture of
$(A \cap E) e(P_1) (A \cap E)$ and $(A \cap E) e(P_1) (A
\cap E)$.

2) If $(P_1,P_2,P_3)$ is generic in
$\C [y]_{\leq d_1} \times \C [y]_{\leq d_2} \times \C
[y]_{\leq d_3}$,
then $(A \cap E) e(P_3) (A \cap E)$ is not a mixture of
$(A \cap E) e(P_1) (A \cap E)$ and $(A \cap E) e(P_2) (A
\cap E)$.

\underline{Proof of 1.}
Define $\phi :  \C [y]_{\leq d_1} \times \C^8  \to \C
[y]_{\leq d_1} \times \C [y]$,
$$\left(  P_1 , (\alpha, \ldots, \theta) \right) \mapsto
(P_1, 
\alpha P_1( \beta y + \gamma) + \delta P_1 ( \epsilon y +
\zeta) + \eta y + \theta).$$
We have $\dim {\rm Im} \, \phi \leq  d_1 +1 + 8 < \dim \C
[y]_{\leq d_1} \times \C [y]_{\leq d_2}$.
If $(P_1, P_2) \in  (\C [y]_{\leq d_1} \times \C [y]_{\leq
d_2}) \setminus {\rm Im} \, \phi$,
it is clear that $(A \cap E) e(P_2) (A \cap E)$ is not a
mixture of
$(A \cap E) e(P_1) (A \cap E)$ and $(A \cap E) e(P_1) (A
\cap E)$.

\underline{Proof of 2.}
Define $\psi :\C [y]_{\leq d_1} \times \C [y]_{\leq
d_2} \times   \C^8 
 \to \C [y]_{\leq d_1} \times \C [y]_{\leq d_2} \times \C
[y]$,
$$\left(  P_1, P_2, (\alpha, \ldots, \theta) \right) \mapsto
(P_1, P_2,
\alpha P_1( \beta y + \gamma) + \delta P_2 ( \epsilon y +
\zeta) + \eta y + \theta).$$
We have $\dim {\rm Im} \, \phi \leq  (d_1 +1) +(d_2+1) +8
< \dim \C [y]_{\leq d_1} \times \C [y]_{\leq d_2} \times \C
[y]_{\leq d_3} $.
If $(P_1, P_2, P_3) \in 
(\C [y]_{\leq d_1} \times \C [y]_{\leq d_2} \times \C
[y]_{\leq d_3}) \setminus {\rm Im} \, \psi$,
it is clear that $(A \cap E) e(P_3) (A \cap E)$ is not a
mixture of
$(A \cap E) e(P_1) (A \cap E)$ and $(A \cap E) e(P_2) (A
\cap E)$.  
\end{proof}

\begin{cor} \label{cor2}
Fix a sequence of integers $d_1,\ldots,d_l \geq 8$.
The generic element  $(P_1, \cdots, P_l)$
of $\prod_{1 \, \leq \, i \, \leq \, l} \C [y]_{\leq \, d_i}$
is an independent sequence.
\end{cor}

Combining Corollaries \ref{cor1} and \ref{cor2}, we get:

\begin{cor} \label{cor3}
Fix a sequence of integers $d_1,\ldots,d_l \geq 8$.
The generic element  $(P_1, \cdots, P_l)$
of $\prod_{1 \, \leq \, i \, \leq \, l} \C [y]_{\leq \, d_i}$
defines a sequence of general and independent polynomials.
\end{cor}

\begin{rem} \label{rem3}
In other words, if $a_i \in A \setminus E$ and $e_i =e(P_i)$ for $1 \leq i \leq l$,
then the automorphism $a_1 e_1 \ldots a_le_l$ satisfies condition $(C2)$.
\end{rem}

\subsection*{C. Explicit examples}
\label{sec:annexC}

Lemmas \ref{criterionusingarithmeticmean} and \ref{lem:indep} below
will allow us to give explicit examples of polynomials $P_1,\ldots,P_l \in \C[y]$
which are general an independent (see Example \ref{concreteexample}).

\begin{lem} \label{criterionusingarithmeticmean}
Let $P \in \C[y]$ be a polynomial of degree $d \geq 3$
and let ${\dis M= - \frac{p_{d-1}}{dp_d}}$ be the arithmetic mean
of its roots. If there exists two consecutive integers $k
\geq 0$
such that $P ^{(k)}(M) \neq 0$, then $P$ is weakly general.
\end{lem}

\begin{proof}
If $P(y)= \alpha P ( \beta y + \gamma)$, then the
automorphism $f$
of the affine line given by $f(y)= \beta y + \gamma$
permutes the roots of $P$.
Since $f$ is affine, we must have $f(M)=M$.
By substituting $M$ for $y$ in the equality
$P^{(k)}(y)=\alpha \beta ^k P^{(k)}(f(y))$,
we get $(1 - \alpha \beta ^k) P^{(k)}(M)= 0$. Whence the
result.  
\end{proof}

\begin{rem} We always have $P^{(d-1)}(M)=0$.
Therefore, if $P$ has degree $2$,
it is not possible to find
two consecutive integers $k$
such that $P ^{(k)}(M) \neq 0$. As a consequence,
it is not possible to show that $P$ is weakly general by using
an analogous version of Lemma  \ref{criterionusingarithmeticmean}.
In fact, it is easy to check that no polynomial of degree $2$
is weakly general!
\end{rem}

\begin{exple} \label{exampleofgeneralpolynomials}
Let $P=\sum_i p_i y^i$ be a polynomial of degree $d \geq 5$.
\begin{enumerate}
\item 
If  $p_{d-1}=0$ and $p_{d-2}p_{d-3} \neq 0$, then $P$ is
general;
\item
If $p_{d-1} \neq 0$ and $p_{d-2}=p_{d-3} = 0$,
then $P$ is general.
\end{enumerate}
\end{exple}

\begin{lem} \label{lem:indep}
A family $(P_i)_i$ of general polynomials
satisfying  
$| \deg \, P_i - \deg \, P_j |> 3$
for any $i \neq j$ is independent.
\end{lem}

\begin{proof}
Let us assume (by contradiction) that
${\dis  \deg \hspace{-1mm} \sum_{1 \,  \leq  \, k \,  \leq 
\, 3 } \hspace{-3mm} \alpha _k P_{i_k} ( \beta_k y +
\gamma_k)
\hspace{2mm} \leq \hspace{2mm} 1}$ and that we do not have
$i_1=i_2=i_3$.\\

\underline{First case.} $i_1,i_2,i_3$ are distinct.

By the assumption, $\deg P_{i_1}, \deg P_{i_2},\deg P_{i_3}$
are distinct,
this is impossible.\\

\underline{Second case.} $i_1,i_2,i_3$ are not distinct.

We may assume that $i_1=i_2 \neq i_3$.

Since $P_{i_1}$ is general, 
for any $\alpha, \beta, \gamma$,
the polynomial  $P_{i_1}(y)- \alpha P_{i_1} (\beta y +
\gamma)$
either has degree $\geq \deg \, P_{i_1} -3$ or is null.
More generally, the same result holds for
${\dis Q(y)=  \sum_{1 \,  \leq  \, k \,  \leq  \, 2 }
\hspace{-3mm} \alpha _k P_{i_1} ( \beta_k y + \gamma_k) }$.
But $| \deg P_{i_3} - \deg P_{i_1} | > 3$ by the assumption,
so that
$\deg Q \neq \deg P_{i_3}$. 
Therefore, we cannot have $\deg (Q +  \alpha _3 P_{i_3} (
\beta_3 y + \gamma_3)) \leq 1$.  
\end{proof}

\begin{exple} \label{concreteexample}
By Example \ref{exampleofgeneralpolynomials}, the polynomial
$y^d + y^{d-1}$ is general for $d \geq 5$.
Therefore, if we set $P_d = y^{4d+1} + y^{4d}$,
the polynomials $P_1, \ldots,P_l$ are general and independent (for any $l$).
As a consequence, if $a_i \in A \setminus E$ and $e_i =e(P_i)$ for $1 \leq i \leq l$,
then $f=a_1 e_1 \ldots a_le_l$ satisfies condition $(C2)$.
If we assume furthermore that $f \in G$ and  $l \geq 7$,
then $<f>_N \neq G$ by Theorem \ref{thm:B}.
\end{exple}

\bibliographystyle{alpha}
\bibliography{biblio}

\end{document}